\documentclass[fleqn,preprint,11pt]{elsarticle}

\makeatletter
\def\ps@pprintTitle{%
 \let\@oddhead\@empty
 \let\@evenhead\@empty
 \def\@oddfoot{\centerline{\thepage}}%
 \let\@evenfoot\@oddfoot}
\makeatother

\usepackage[margin=3cm]{geometry}    

\usepackage{amsmath,amssymb,amsthm}
\usepackage{epsfig,graphics}
\usepackage{comment}
\usepackage{url}

\graphicspath{{./figures/}}

\usepackage{tabularx}

\usepackage{array,booktabs}
\usepackage{siunitx}
\usepackage{algorithm}
\usepackage{algorithmicx}

\usepackage{color}

\newcommand{\Bk}{\color{black}}

\RequirePackage{xr-hyper}[6.00]
\RequirePackage{ifpdf}[2.3]
\RequirePackage{hyperref}[6.83]
\RequirePackage{xcolor}[2.11]
\colorlet{inlinkcolor}{green!50!black}
\colorlet{exlinkcolor}{red!50!black}
\hypersetup{
  colorlinks = true,
  allcolors = inlinkcolor,
  urlcolor = exlinkcolor,
}

\RequirePackage[capitalize,nameinlink]{cleveref}[0.19]

\crefname{section}{Section}{sections}
\crefname{subsection}{subsection}{subsections}
\Crefname{section}{Section}{Sections}
\Crefname{subsection}{Subsection}{Subsections}

\Crefname{figure}{Figure}{Figures}

\crefformat{equation}{\textup{#2(#1)#3}}
\crefrangeformat{equation}{\textup{#3(#1)#4--#5(#2)#6}}
\crefmultiformat{equation}{\textup{#2(#1)#3}}{ and \textup{#2(#1)#3}}
{, \textup{#2(#1)#3}}{, and \textup{#2(#1)#3}}
\crefrangemultiformat{equation}{\textup{#3(#1)#4--#5(#2)#6}}%
{ and \textup{#3(#1)#4--#5(#2)#6}}{, \textup{#3(#1)#4--#5(#2)#6}}{, and \textup{#3(#1)#4--#5(#2)#6}}

\Crefformat{equation}{#2Equation~\textup{(#1)}#3}
\Crefrangeformat{equation}{Equations~\textup{#3(#1)#4--#5(#2)#6}}
\Crefmultiformat{equation}{Equations~\textup{#2(#1)#3}}{ and \textup{#2(#1)#3}}
{, \textup{#2(#1)#3}}{, and \textup{#2(#1)#3}}
\Crefrangemultiformat{equation}{Equations~\textup{#3(#1)#4--#5(#2)#6}}%
{ and \textup{#3(#1)#4--#5(#2)#6}}{, \textup{#3(#1)#4--#5(#2)#6}}{, and \textup{#3(#1)#4--#5(#2)#6}}

\crefdefaultlabelformat{#2\textup{#1}#3}

\DeclareMathOperator{\sign}{sign}
\newcommand{\eps}{\varepsilon}

\newcommand{\be}{\begin{equation}}
\newcommand{\ee}{\end{equation}}
\newcommand{\ba}{\begin{aligned}}
\newcommand{\ea}{\end{aligned}}
\newcommand{\bea}{\begin{eqnarray}}
\newcommand{\eea}{\end{eqnarray}}

\def\tx{{\tilde{x}}}
\def\ty{{\tilde{y}}}
\def\txi{{\tilde{\xi}}}
\def\teta{{\tilde{\eta}}}
\def\td{{\tilde{d}}}

\def\C{\mathcal{C}}
\def\N{\mathcal{N}}
\def\F{\mathcal{F}}

\def\ltfar{\Lambda_{far}}
\def\lt2{\Lambda_{far}}
\def\lt{{\Lambda}}

\def\NUFFTone{{\bf N}}
\def\NUFFTtwo{{\bf N}^\ast}

\def\f{{\bf f}}

\def\y{{\bf y}}

\def\bc{{\bf c}}
\def\bl{{\bf l}}
\def\bn{{\bf n}}
\def\bq{{\bf q}}
\def\bs{{\bf s}}
\def\bt{{\bf t}}
\def\bu{{\bf u}}
\def\bv{{\bf v}}
\def\bw{{\bf w}}

\def\bD{{\bf D}}
\def\bI{{\bf I}}
\def\bL{{\bf L}}

\def\bR{{\bf R}}
\def\bS{{\bf S}}
\def\bT{{\bf T}}

\def\G{{\bf G}}
\def\K{{\bf K}}
\def\P{{\bf P}}

\def\eh{{\hat{\bf e}}}

\def\half{\frac{1}{2}}

\newtheorem{theorem}{Theorem}
\newtheorem{lemma}{Lemma}
\newtheorem{corollary}{Corollary}

\theoremstyle{definition}
\newtheorem{definition}{Definition}
\newtheorem{remark}{Remark}
\newtheorem{claim}{Claim}

\def\ccm{Center for Computational Mathematics, Flatiron Institute, Simons Foundation,
  New York, New York 10010}

\def\njit{Department of Mathematical Sciences,
New Jersey Institute of Technology,
Newark, New Jersey 07102}

\def\nyu{Courant Institute of Mathematical Sciences,
  New York University, New York, New York 10012}

\def\papertitle{Periodic Fast Multipole Method}

\begin{document}

\begin{frontmatter}

\title{\papertitle}


\author{Ruqi Pei\fnref{njit}}
\address[njit]{\njit}
\ead{rp696@njit.edu}

\author{Travis Askham\fnref{njit}}
\ead{askham@njit.edu}

\author{Leslie Greengard\fnref{ccm,nyu}}
\address[ccm]{\ccm}
\address[nyu]{\nyu}
\ead{greengard@courant.nyu.edu}

\author{Shidong Jiang\fnref{ccm,njit}}
\ead{sjiang@flatironinstitute.org}

\begin{abstract}
A new scheme is presented for imposing 
periodic boundary conditions on unit cells with arbitrary 
source distributions. We restrict our attention here to 
the Poisson, modified Helmholtz, Stokes and modified Stokes
equations. The approach
extends to the oscillatory equations of mathematical physics,
including the Helmholtz and Maxwell equations, but we will 
address these in a companion paper, since the nature of the problem
is somewhat different and includes the consideration of 
quasiperiodic boundary conditions and resonances.
Unlike lattice sum-based methods, the scheme is insensitive to 
the unit cell's aspect ratio and is easily coupled to adaptive 
fast multipole methods (FMMs). Our analysis relies on classical
``plane-wave" representations of the fundamental solution, 
and yields an explicit low-rank representation of the field due to 
all image sources beyond the first layer of neighboring unit cells.
When the aspect ratio of the unit cell is large, our scheme
can be coupled with the nonuniform 
fast Fourier transform (NUFFT) to accelerate the evaluation of the 
induced field. Its performance is illustrated
with several numerial examples.
\end{abstract} 

\begin{keyword}
  periodic boundary conditions \sep
  fast multipole method \sep
  plane wave representation \sep
  nonuniform fast Fourier transform \sep
  low rank factorization \sep
  multipole expansion \sep
  Poisson equation \sep 
  modified Helmholtz equation \sep
  Stokes equations \sep
  modified Stokes equations
\end{keyword}

\end{frontmatter}
\section{Introduction}\label{intro}
Applications in electrostatics, magnetostatics, fluid mechanics, and 
elasticity often involve sources contained in a unit cell $\C$, centered
at the origin, on which are imposed periodic boundary conditions.
In two dimensions, 
such a unit cell is defined by two fundamental translation vectors
$\eh_1$ and $\eh_2$. 
In the doubly periodic setting, we assume (without loss of generalilty) 
that $\| \eh_1 \| \geq \| \eh_2 \|$ and that, by a suitable rotation,
$\eh_1$ is aligned with the $x$-axis and $\eh_2$ lies in the upper
half space (see \cref{fig1}).
That is, we let
$\C=\left\{x_1\eh_1+x_2\eh_2\in 
\mathbb{R}^2|\ x_1, x_2\in [-\half,\half]\right\}$,
where $\eh_1=\langle d,0 \rangle$,
$\eh_2=\langle \xi,\eta \rangle$, with
$d, \eta >0$ and $d \ge \sqrt{\xi^2+\eta^2}$. 
In the singly periodic setting, we assume that the periodic direction
is aligned with the $x$-axis, but can no longer assume that 
$\| \eh_1 \| \geq \| \eh_2 \|$. Without loss of generality, however, 
we can assume that the unit cell is rectangular (\cref{fig1}) and
of dimension $d \times \eta$.
Letting $\bt = (x,y)$ and letting $u(\bt)$ denote a scalar 
quantity of interest,
by {\em doubly periodic}
boundary conditions we mean that $u(\bt)$ must satisfy:
\be
\ba
u(\bt+\eh_1)&= u(\bt),\\
u(\bt+\eh_2)&= u(\bt).
\ea
\label{pbc}
\ee
By {\em singly periodic}
boundary conditions we mean that $u(\bt)$ must satisfy:
\be
\ba
u(\bt+\eh_1)&= u(\bt),\\
\ea
\label{pbc1}
\ee
and a standard outgoing/decay condition in the $y$-direction.

\begin{figure}[t]
\centering
\includegraphics[width=100mm]{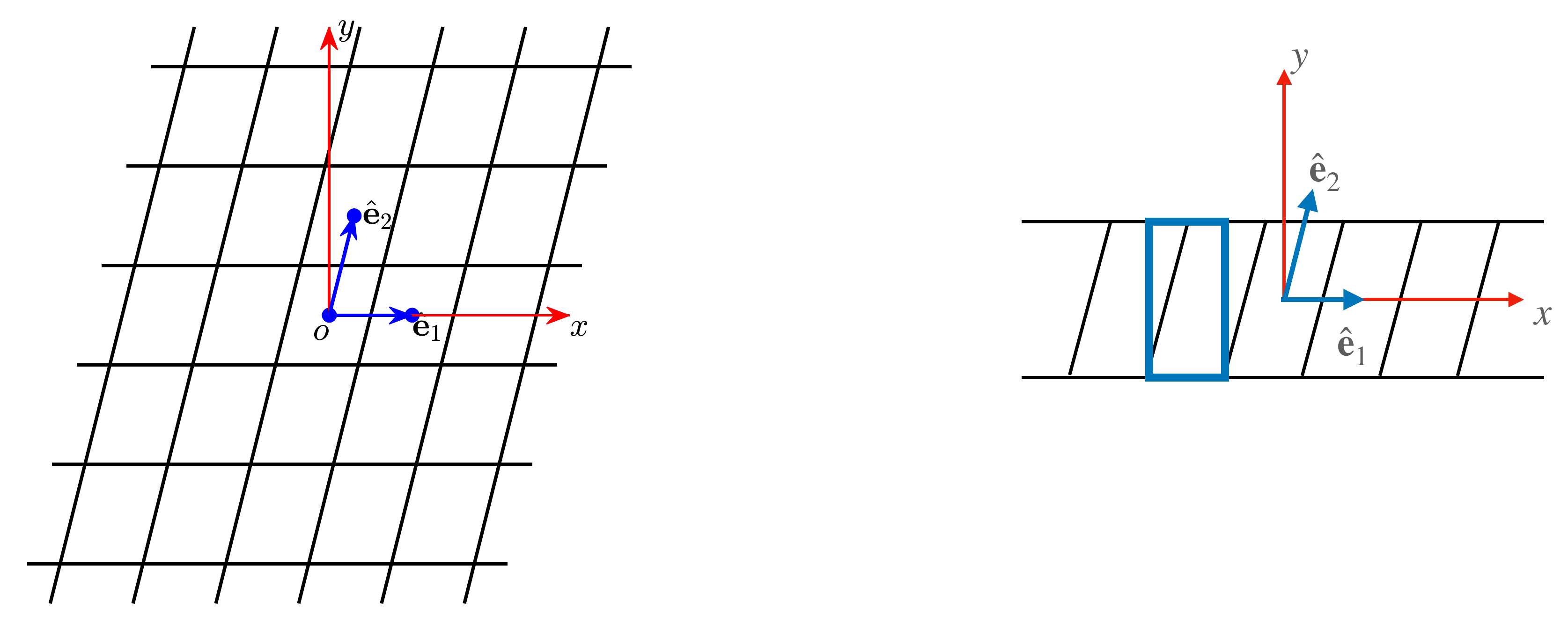}
\caption{
In the doubly periodic case (left), the unit cell is a parallelogram
which tiles the entire plane. By convention, we assume that the lattice is 
oriented so that the longer cell dimension is aligned with the $x$-axis:
$\| \eh_1 \| \geq \| \eh_2 \|$.
In the singly periodic case (right), we assume the periodic direction
is aligned with the $x$-axis. The unit cell may still be 
a parallelogram, but we can always define a corresponding 
rectangular unit cell, indicated by thick blue lines. 
In this case, we cannot assume 
that the long cell dimension is aligned with the $x$-axis.
Periodic boundary conditions are imposed through the method of images:
that is, 
by including the influence of the translated sources in every image cell on
the targets in the fundamental unit cell.}
\label{fig1}
\end{figure}

For the moment, let us assume that the
governing partial differential equation (PDE) is 
\begin{equation}
\Delta u(\bt) - \beta^2 u(\bt) = \sum_{j=1}^{N_S} q_j \delta(\bt-\bs_j),
\label{govpde}
\end{equation}
with $\beta$ real and non-negative.
Here, $\bt,\bs_j$ are points lying within the unit cell $\C$.
We refer to \eqref{govpde} as the modified Helmholtz equation when
$\beta > 0$.  When $\beta = 0$, of course, we obtain
the Poisson equation.
In two dimensions, the free-space Green's functions for these equations
are well-known and given by \cite{mikhlin,stakgold}
\[ 
G(\bt,\bs) = \frac{1}{2\pi} K_0(\beta\|\bt-\bs\|), \quad
G(\bt,\bs) = \frac{1}{2\pi} \log(1/\|\bt-\bs\|),
\]
respectively, where $K_0$ denotes the zeroth order modified Bessel function 
of the second kind \cite{nisthandbook}.

Thus, in free space, the solution to \eqref{govpde} at targets
$\bt_1,\dots,\bt_{N_T}$ is given by
\be\label{fsum}
u^{(f)}(\bt_i) =\sum_{j=1}^{N_S} G(\bt_i, \bs_j) \, q_j,
\qquad i=1,\ldots, N_T.
\ee
where $G(\bt,\bs)$ is the relevant free-space Green's function.
It is well-known that algorithms such as 
the fast multipole method (FMM)
\cite{cheng1999jcp,greengard1987jcp,greengard1997actanum,ying2004jcp} 
reduce the computational cost of evaluating \eqref{fsum} 
from $O(N_S\cdot N_T)$ to $O(N_S+N_T)$, with
the prefactor depending logarithmically on the desired precision. 

Since the problem at hand is classical, there are many approaches now
available for imposing periodicity. Three common approaches are:
direct discretization of the governing PDE including boundary conditions
to yield a large sparse linear system of equations, spectral methods which 
solve \eqref{govpde} using Fourier analysis, and the method of images,
based on tiling the plane with copies of the unit cell and computing the 
formal solution:
\be\label{psum}
u(\bt_i)
=\sum_{j=1}^{N_S} K^{(p)}(\bt_i, \bs_j) q_j \, ,
\ee
where
\[
K^{(p)}(\bt, \bs) 
= \sum_{(m,n)\in \mathbb{Z}^2} 
G(\bt,\bs+\bl_{mn}) 
\]
denotes the periodic Green's function.
Here, $\mathbb{Z}^2=\{(m,n)|m,n\in \mathbb{Z}\}$ is the set of 
integer lattice points in the plane
and $\bl_{mn}=m\eh_1+n\eh_2$. It is straightforward
to verify that this formal solution satisfies the PDE and the 
boundary conditions.
For the modified Helmholtz equation, the series defining
$K^{(p)}(\bt, \bs)$ is convergent and 
requires no further discussion. 
For the Poisson equation, the series is conditionally convergent but
straightforward to interpret if the net charge $\sum_{j=1}^{N_S} q_j = 0$.

Without entering into a detailed review of the literature, 
we note that the spectral approach is standard in 
solid-state physics and quantum mechanics and attributed to 
Ewald \cite{ewald} and Bloch \cite{bloch} (with earlier 
work in the mathematics literature by Floquet, Hill and others).
We focus here on the method of images, using \eqref{psum},
which is more common in acoustics, electromagnetics,
and fluid dynamics and dates back to Lord Rayleigh 
\cite{rayleigh}.

\begin{definition}
In the doubly periodic case,
we decompose the two-dimensional integer lattice $\mathbb{Z}^2$ 
into 
$\Lambda_{near} =\{(m,n)| m \in \{-m_0,\ldots,m_0 \} ,n \in \{ -1,0, 1 \} \}$ 
and 
$\Lambda_{far} = \mathbb{Z}^2-\Lambda_{near}$.
$m_0=1$ is sufficient for rectangular unit cells. In order to
allow for parallelograms with arbitrarily small angles 
($\xi \gg \eta$), it is sufficient to set $m_0=3$. 
The region covered by the unit cell $\C$ and its 
nearest images, indexed by $\Lambda_{near}$, will be referred
to as the near field and denoted by $\N$.
The region covered by the remaining image cells,
indexed by $\Lambda_{far}$, will be referred
to as the far field and denoted by $\F$. For consistency in notation,
in the singly periodic case, we define
$\Lambda^{(1)}_{near} =\{ (m,0) | m \in \{-1,0, 1 \} \}$ 
and 
$\Lambda^{(1)}_{far} = \{ (m,0) | m \in \mathbb{Z} \} - \Lambda^{(1)}_{near}$.
\end{definition}

\begin{definition}  \label{aspectdef}
In the two-dimensional case, we define the
aspect ratio of the fundamental unit cell by
$A=d/\eta$. Since we have chosen to orient the longer 
lattice vector $\eh_1$ with the $x$-axis, $A \geq 1$.
The problem is computationally more involved when $A$ is large.
In the one-dimensional case, we define the
aspect ratio by $A=\max(1,\eta/d)$
As we shall see below,
it is again when $A$ is large that the computation 
is most difficult.
(See \cref{fig1}.)
\end{definition}

It is useful to express 
the singly or doubly periodic Green's function in the form:
\[
 K^{(p)}(\bt,\bs) 
= K^{near}(\bt,\bs) \ + \  K^{far}(\bt,\bs),
\]
where
\be
\ba 
\label{nearfarsplit}
K^{near}(\bt,\bs)
&=
\sum_{(m,n) \in \Lambda_{near} }
G(\bt,\bs+\bl_{mn}) \\
K^{far}(\bt,\bs) &=
\sum_{(m,n) \in \Lambda_{far} }
G(\bt,\bs+\bl_{mn}) 
\ea
\ee
Because the sources in 
$K^{far}(\bt,\bs)$ are distant,
it is possible to express their contributions within the
unit cell as a series
\be
\label{latsum1}
K^{far}(\bt,\bs) =
\sum_{l=-\infty}^{\infty} S_l I_l(\lambda \| \bt-\bs \|) 
e^{il\theta_{\bt,\bs}}
\ee
with $\theta_{\bt,\bs} = \arg(\bt-\bs)$,
where 
$I_l$ denotes the modified Bessel function of the first kind
\cite{nisthandbook} and $S_l$ denotes the {\em lattice sum}
\be
\label{latsumdef}
 S_l = \sum_{\substack{ (m,n)\in \Lambda_{far}}} 
K_l(\lambda|\bl_{mn}|) e^{il\phi_{mn}},
\ee
with $\phi_{nm} = \arg(\bl_{mn})$. 
(This is a straightforward application of the Graf addition 
theorem \cite[\S 10.23]{nisthandbook}.)
The reason for omitting the nearest image cells from 
$K^{far}(\bt,\bs)$
is that the convergence behavior of the series expansion  
in \eqref{latsum1}
is controlled by the distance of the nearest source 
from the disk centered at the origin and enclosing the unit cell
(see \cref{fig2}).
The more images included in the near field, the faster the convergence
rate of the local expansion.

When the unit cell is square or has an aspect ratio near to
one, this yields an optimal scheme and is widely used in 
periodic versions of the fast multipole method 
\cite{berman1994jmp,greengard1987jcp,biros2015cicp}. 
Of particular note is 
\cite{yan2018} which extends a three-dimensional version of the 
kernel-independent FMM library \cite{biros2015cicp} to permit 
the imposition of periodicity on the unit cube
in one, two or three directions. 
(See
\cite{berman1994jmp,denlinger2017jmp,dienstfrey2001prsl,enoch2001,huang1999jmp,linton2010,mcphedran2000jmp,moroz2006,otani2008} 
for further discussion and references,
largely in the context of the Poisson, Helmholtz and Maxwell equations.)
Unfortunately, lattice sum-based approaches are less efficient
when the unit cell has high aspect ratio, as
illustrated for a doubly periodic problem in \cref{fig2}. 
The difficulty is that every source assigned to the far field must
be in the exterior of the smallest disk enclosing the unit cell
in order to ensure convergence of the local expansion.
This may require redefining $\Lambda_{far}$ to exclude a large
number of image cells, redefining $\Lambda_{near}$ to include
those image cells, and a major
modification of the underlying fast algorithm.

\begin{remark}
In the FMM, 
lattice sums are not used for the evaluation of 
$K^{far}(\bt,\bs)$ for each source and target. Instead, given
a multipole expansion for the unit cell, one constructs a single
local expansion of the form
\[ \sum_{l=-\infty}^{\infty} \alpha_l I_l(\lambda \| \bt \|) 
e^{il\theta_{\bt}}
\]
that captures the field due to all sources in the far 
field $\F$ within the unit cell.
This is a slight
modification of Rayleigh's original method \cite{rayleigh}.
The coefficients $\alpha_l$ are determined from the multipole coefficients
through a formula which involves the lattice sums $S_l$
(see above references).
\end{remark}

\begin{figure}[t]
\centering
\includegraphics[width=120mm]{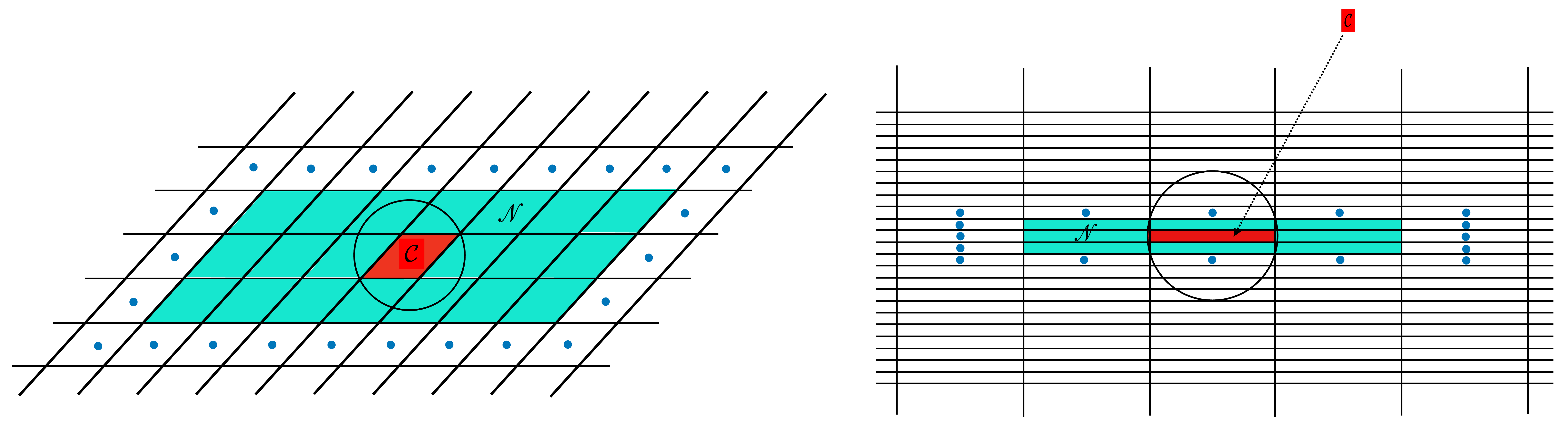}
\caption{Two fundamental unit cells $\C$ in the doubly periodic case. 
On the left, the indicated $7 \times 3$ grid of neighbors define the 
near region $\N$
when the parameter $m_0 = 3$. On the right, when $m_0=1$, the near region 
corresponds to the $3 \times 3$ grid of neighbors, which is sufficient
for rectangular lattices.
We also plot the centers of the nearest image cells outside $\N$.
Note that if
the field due to distant images is represented in the unit cell by a
Taylor series, the convergence behavior is controlled by the distance
from the smallest disk covering the unit
cell $\C$ to the nearest such image, which must lie {\em outside} the 
disk. For the geometry on the left, all image sources 
outside $\N$ satisfy this
constraint and the Taylor series converges. For high aspect ratio cells,
illustrated on the right, several of the images lie within the disk, and
a region much larger than $\N$ must be excluded for the corresponding 
Taylor series to be convergent.
}
\label{fig2}
\end{figure}

Recently, two new approaches were developed that carry out 
a free space calculation of the form 
\eqref{fsum} over sources in $\Lambda_{near}$ and correct for the 
lack of periodicity using 
an integral representation 
\cite{barnett2010jcp,barnett2011bit} 
or a representation in terms of discrete 
auxiliary Green's
functions \cite{barnett2018cpam,liu2016jcp,wang2021}.
Both of these approaches are effective even for high aspect ratio
unit cells, but require the solution of a possibly ill-conditioned 
linear system of equations in the correction step.

In this paper, we develop a new scheme to treat 
periodic boundary conditions based on
an explicit, low-rank representation 
for the influence of all distant sources in the far field
(those in image cells indexed by $\Lambda_{far}$). It
avoids the lattice sum/Taylor series formalism
altogether and is insensitive to the aspect ratio of the
unit cell.  
It was motivated by, and makes use of, the fast algorithms
for lattice sums and elliptic functions developed in 
\cite{dienstfrey2001prsl,enoch2001,huang1999jmp,mcphedran2000jmp} 
and the fast translation operators
used in modern versions of the FMM 
\cite{cheng2006jcp,greengard1997actanum,hrycak1998}.

\begin{figure}[t]
\centering
\includegraphics[width=100mm]{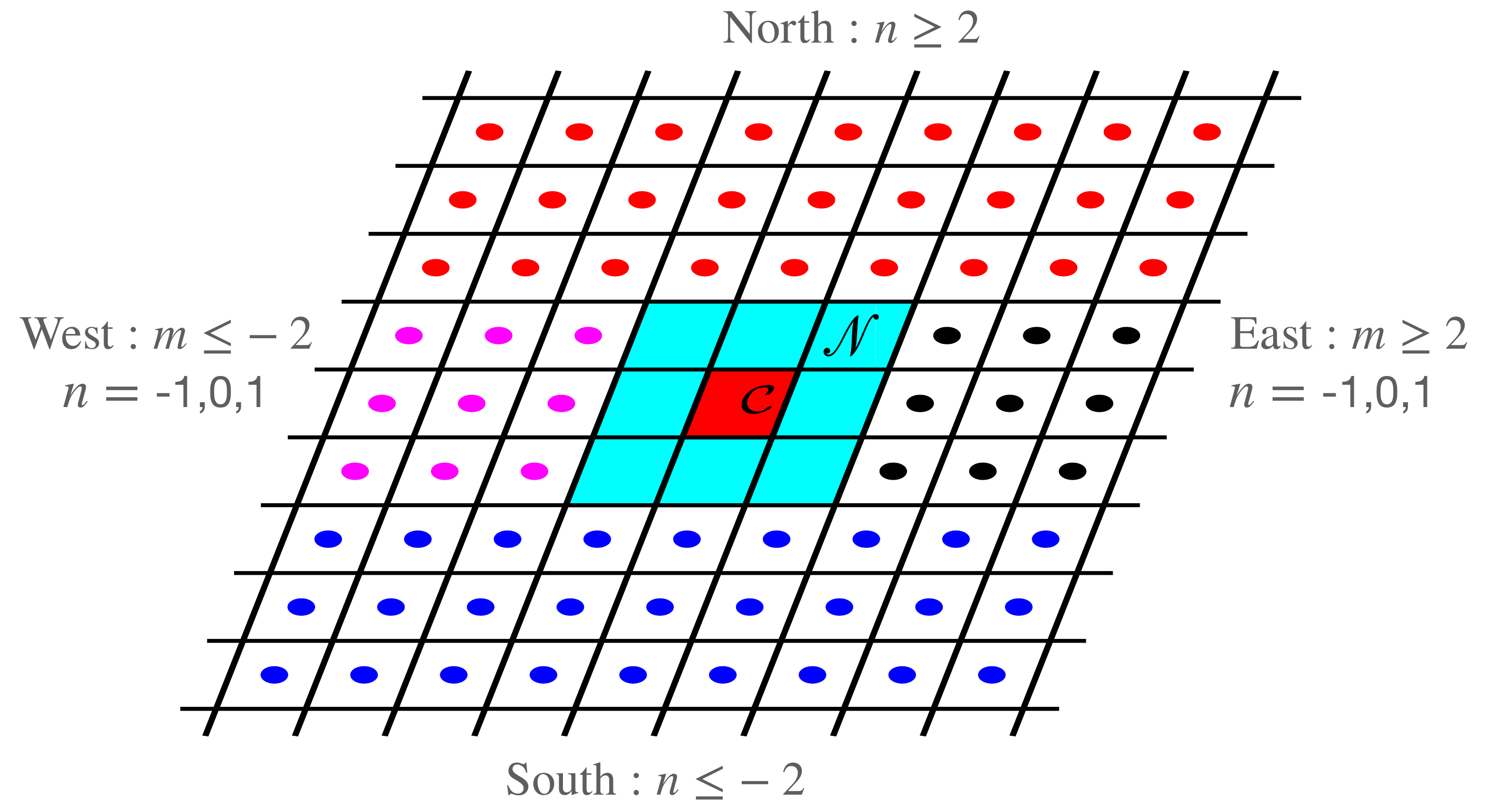}
\caption{The tiling of the plane in the far region (for doubly
periodic problems with $m_0=1$) can be decomposed
into four parts. The fundamental unit cell 
is indicated by ${\cal C}$ and the near field by 
${\cal N}$. All other copies of the unit cell lie to the
``south" (blue, with $n \leq -2$), the 
``north" (red, with $n \geq 2$), the 
``west" (magenta, with $n = -1,0,1$ and $m < -1$) or the
``east" (black, with $n = -1,0,1$ and $m > 1$).
}
\label{fig3}
\end{figure}

The essence of the approach is
easily illustrated in the doubly periodic setting 
(\cref{fig3}), where the tiling of the far field is divided
into four subregions. 

\begin{remark}
To fix notation, we will denote by $\bt = (x,y)$
or $\bt_l = (x_l,y_l)$ the coordinates of a target point where we
seek to evaluate the field.
We will denote by $\bs = (x',y')$
or $\bs_j = (x'_j,y'_j)$ the coordinates of a source point.
\end{remark}

\begin{claim}
Consider the field induced by all sources lying in image cells 
with centers lying to the ``south":
$\{(m,n)\in \mathbb{Z}^2| m \in \mathbb{Z},\ n\leq -2\}$
(\cref{fig3}).
Then, for any target $\bt \in \C$,
\be\label{lowranksouth}
\ba
u(\bt)
&=\sum_{j=1}^{N_S}\left(\sum_{m=-\infty}^\infty \sum_{n=-\infty}^{-2}
G(\bt,\bs_j+\bl_{m,n}) \right)q_j \\
&= 
\sum_{m=-\infty}^\infty 
c_m\ 
e^{-\sqrt{(2\pi m/d)^2 + \beta^2}y} e^{2 \pi i m x/d} \, ,
\ea
\ee
where 
\[ c_m =  \frac{\pi}{d} \left(
\frac{
e^{-2 \sqrt{( 2 \pi m/d)^2 + \beta^2} \eta}
e^{-  4\pi i m \xi/d}}
{1 - e^{-\sqrt{( 2 \pi m/d)^2 + \beta^2} \eta}
e^{-2\pi i m \xi/d}} \right)
 \, 
\sum_{j=1}^{N_S} q_j \frac{e^{\sqrt{(2\pi m/d)^2 + \beta^2}y'_j} \,
e^{-2 \pi i m x'_j/d}}{\sqrt{( 2 \pi m/d)^2 + \beta^2}}
 \, .
\]
\end{claim}

We postpone the derivation of this formula to 
\cref{sec:yukawaterms}, but let us briefly examine its consequences.
First, the behavior of the 
series is not controlled by a radius of convergence, 
as it is for methods based on
lattice sums and Taylor series. 
Second, the series converges exponentially fast. Since $|y|,|y_j| \leq 
\frac{\eta}{2}$, it is easy to see that the $m$th term of the series decays
faster than $e^{-2 \pi m \eta/d}$. 
Clearly, 
approximately $6 A$ terms yields 
double precision accuracy, where $A=d/\eta$ is the aspect ratio of the unit cell.
Computing the moments $c_m$ requires 
$O(N_S \, A)$ work. Subsequent evaluation at
$N_T$ target points again requires $O(N_T \, A)$ work.
In short, this is an efficient low-rank, separable 
representation of the potential due to a subset of the image sources, 
the rank of which 
grows at most linearly with $A$.
When $A$ is sufficiently large, 
we will show how to use the non-uniform FFT 
(NUFFT)~\cite{finufftlib,nufft2,nufft3,nufft6,nufft7}
to obtain an algorithm whose cost is of the order
$O(\log(1/\epsilon)(A\log A +(N_S+N_T)\log(1/\epsilon)))$. 

Below, we complete and generalize the representation 
\eqref{lowranksouth}, permitting the imposition of
periodic boundary conditions in one or two directions
for a variety of non-oscillatory PDEs in the plane.
We will denote by
\be
\label{k2def}
K_2^{far}(\bt,\bs)=\sum_{(m,n)\in \ltfar} G(\bt,\bs+\bl_{mn})
\ee
the far-field kernel for doubly periodic problems.
In the singly periodic case, 
we denote the corresponding far-field kernel by
\be
\label{k1def}
K_1^{far}(\bt,\bs)=\sum_{(m,0)\in \ltfar^{(1)}} G(\bt,\bs+\bl_{m0}) \; .
\ee

\begin{definition} \label{peropdef}
Let 
$\bS = \{ \bs_j \, | \, j = 1,\dots,N_S \}$ and 
$\bT = \{ \bt_l \, | \, l = 1,\dots,N_T \}$ 
denote collections of sources and targets, respectively,
in the unit cell $\C$ and let 
$\bq = (q_1,\dots,q_{N_S})$ denote a vector of
``charge" strengths.
With a slight abuse of notation, 
we define the $N_T \times N_S$ {\em periodizing operators}
$\P_1 = \P_1^\C(\bT,\bS)$ and $\P_2 = \P_2^\C(\bT,\bS)$ by 
\[ \P_1(l,j) = K_1^{far}(\bt_l,\bs_j) 
\]
and
\[ \P_2(l,j) = K_2^{far}(\bt_l,\bs_j) ,
\]
where 
the far field kernels $ K_1^{far}$ and $ K_2^{far}$
are given by \eqref{k2def} and \eqref{k1def}.
The vectors 
\[ \P_1 \bq, \quad \P_2 \bq 
\]
will be referred to as the {\em periodizing potentials}.
We will omit the superscript and depedence on source and target
locations when it is clear from context.
We also assume that the governing PDE is clear from context.
\end{definition}

It is worth noting that our
method yields an explicit, low-rank representation of the 
periodic Green's function, without the need to solve any
auxillary linear systems. 
In fact, the periodizing operators $\P_2,\P_1$ 
admit plane-wave factorizations of rank $O(A)$ similar to the
formula for the ``south'' images described above,
leading to simple fast algorithms for their evaluation.
More precisely, letting $r = O(A)$ be the 
numerical rank of $\P_2$ or $\P_1$ to precision $\eps$,
a simple, direct method requires
$O(r(N_S+N_T))$ work. When $r$
is large, a more elaborate
algorithm using the NUFFT requires only
$O(\log(1/\epsilon)(r\log r+(N_S+N_T)\log(1/\epsilon)))$ work. 

In \cref{sec:analysis}, we review the mathematical and computational 
foundations of the 
method. In \cref{sec:yukawaterms}, we discuss the modified Helmholtz
case in detail. In 
\cref{sec:lap}, we discuss the Poisson equation, where charge
neutrality is a necessary constraint.
The Stokes and modified Stokes problems are considered in 
\cref{sec:stokes}.
In \cref{sec:fastalg}, we describe the full scheme including
NUFFT acceleration. We illustrate the performance of the 
method in \cref{sec:results} with several
numerical examples and describe future
extensions of the method in \cref{sec:conclusions}.
An extension of our representation for multipole sources is provided in
\ref{sec:yukawamultipole} and \ref{sec:laplacemultipole}.
\section{Mathematical preliminaries}  \label{sec:analysis}

In this section, we summarize the main mathematical tools used in
deriving our low-rank representation. These include the Poisson
summation formula, Sommerfeld integral representations for free-space
Green's functions, the nonuniform FFT, and high order accurate 
quadrature schemes.

\subsection{The Poisson summation formula}
Let $f(x): \mathbb{R} \rightarrow \mathbb{C}$ be a function defined
on the real line which has a well-defined Fourier transform
\be
\hat{f}(k) = \int_{\mathbb{R}}e^{-ikx}f(x)dx ,
\ee
for which the Fourier inversion theorem holds. That is,
\[
 f(x) = \frac{1}{2\pi}\int_{\mathbb{R}}e^{ikx}\hat{f}(k)dk.
\]
The Poisson summation formula (see, for example, \cite{mckean}) 
then states that
\be
\sum_{n=-\infty}^\infty f\left(x+\frac{2\pi n}{h}\right) = \frac{h}{2\pi}\sum_{m=-\infty}^\infty \hat{f}(mh)e^{imhx}.
\ee
This holds for a broad class of functions, and extends to distributions
such as the Dirac delta function.
In the latter case, we have~\cite{jones}
\be\label{deltarep}
\sum_{n=-\infty}^\infty \delta\left(x+\frac{2\pi n}{h}\right) =
\frac{h}{2\pi}\sum_{m=-\infty}^\infty e^{imhx}.
\ee
\subsection{Plane wave representations for the modified Helmholtz equation}
The Green's function for the modified Bessel function,
$K_0(\beta \sqrt{x^2+y^2})$,
and its higher order multipole terms
have the well-known
Sommerfeld integral representation~\cite{morse,mcphedran2000jmp,nisthandbook}:
\be
\label{k0planewave}
K_l\left(\beta\sqrt{x^2+y^2}\right) e^{i l \phi}
=\left\{\begin{aligned}
&\frac{i^l}{2 \beta^l} \int_{-\infty}^\infty
  \left(\sqrt{\lambda^2+\beta^2}-\lambda\right)^l
  \frac{e^{-\sqrt{\lambda^2+\beta^2}y}}{\sqrt{\lambda^2+\beta^2}}
  e^{i\lambda x} d\lambda, \quad y>0,\\
&\frac{(-i)^l}{2 \beta^l} \int_{-\infty}^\infty
  \left(\sqrt{\lambda^2+\beta^2}+\lambda\right)^l
  \frac{e^{\sqrt{\lambda^2+\beta^2}y}}{\sqrt{\lambda^2+\beta^2}}
  e^{i\lambda x} d\lambda, \quad y<0, \\
&\frac{1}{2 \beta^l} \int_{-\infty}^\infty
  \left(\sqrt{\lambda^2+\beta^2}+\lambda\right)^l
  \frac{e^{-\sqrt{\lambda^2+\beta^2}x}}{\sqrt{\lambda^2+\beta^2}}
  e^{i\lambda y} d\lambda, \quad x>0,\\
&\frac{(-1)^l}{2 \beta^l} \int_{-\infty}^\infty
  \left(\sqrt{\lambda^2+\beta^2}-\lambda\right)^l
  \frac{e^{\sqrt{\lambda^2+\beta^2}x}}{\sqrt{\lambda^2+\beta^2}}
  e^{i\lambda y} d\lambda, \quad x<0,
\end{aligned}
\right.
\ee
for $l \geq 0$.

\subsubsection{Plane wave representations for the Poisson equation}

The plane-wave expansion of the Green's function for the 
Laplacian is typically invoked for the complex analytic
function 
\[ \frac{1}{z} =  \frac{1}{x+iy} = 
\frac{x-iy}{x^2 + y^2} =
 - \left(\frac{\partial}{\partial x}  
 - i \, \frac{\partial}{\partial y} \right) \log(1/\sqrt{x^2+y^2}) \, ,
 \]
rather than $\log(1/\sqrt{x^2+y^2})$ itself. 
This is sufficient for our 
purposes, where we assume the collection of sources in the unit cell
$\C$ satisfies charge neutrality:
\be\label{chargeneutral}
\sum_{j=1}^{N_S} q_j=0.
\ee

In \ref{sec:laplacemultipole}, we will consider multipole sources as well
as charge distributions, and will make use of the representations:
\be\label{laplacepw}
\frac{1}{z^l}=
\left\{\begin{aligned}
&\frac{1}{(l-1)!} \, \int_0^\infty \lambda^{l-1} e^{-\lambda z} d\lambda, \quad x>0,\\
&\frac{(-1)^l}{(l-1)!}\, \int_0^\infty \lambda^{l-1} e^{\lambda z} d\lambda, \quad x<0,\\
&\frac{(-i)^l}{(l-1)!}\, \int_0^\infty \lambda^{l-1} e^{i\lambda z} d\lambda, \quad y>0,\\
&\frac{i^l}{(l-1)!}\, \int_0^\infty \lambda^{l-1} e^{-i\lambda z} d\lambda, \quad y<0,
\end{aligned}
\right.
\ee
for $l \geq 1$.
These representations are useful in
developing diagonal translation operators for 
FMMs \cite{cheng2006jcp,hrycak1998} as well as for the computation of harmonic
lattice sums and elliptic functions
\cite{huang1999jmp}.
Note that the integrals in \cref{laplacepw} are consistent with
\eqref{k0planewave}. In the case $l=1$, for example, 
this requires taking the limit $\beta\rightarrow 0$ and using
the formula~\cite[\S10.30.2]{nisthandbook}:
\be\label{k1limitform}
\lim_{\beta \rightarrow 0} \beta K_1(\beta r)e^{-i\theta}=\frac{1}{z}, \quad z=re^{i\theta}.
\ee

\subsubsection{Generalized Gaussian quadrature for the doubly periodic case}
\label{sec:ggq}

In evaluating the integrals in
\eqref{k0planewave} or \eqref{laplacepw}, we will require suitable
quadrature rules. 
More generally, we would like efficient rules 
of the form
\begin{eqnarray*}
& & \int_{0}^\infty 
\frac{e^{-\sqrt{\lambda^2+\beta^2}x}}{\sqrt{\lambda^2 + \beta^2}}
[ M_1(\lambda) e^{i\lambda y} + {M_1(-\lambda)} 
e^{-i\lambda y} ] \, d\lambda \\
& & \approx 
 \sum_{k=1}^{N} \frac{e^{-\sqrt{\lambda_k^2+\beta^2}x}}{\sqrt{\lambda_k^2 + \beta^2}}
[ M_1(\lambda_k) e^{i\lambda_k y} + M_1(-\lambda_k) e^{-i\lambda_k y} ] \, w_k
\end{eqnarray*}
and
\[
\int_0^\infty  e^{-\lambda z}  M(\lambda) d\lambda 
\approx 
\sum_{k=1}^{N} e^{-\lambda_k z }
M(\lambda_k)  w_k
\]
for $(x,y)$ in a bounded domain of $\mathbb{R}^2$.
The functions $M_1(\lambda)$ and $M(\lambda)$ here are smooth functions
of $\lambda$ that depend on the source locations and strengths
and are derived from the infinite series that appear
in the periodizing operators. 
The remaining cases in 
\eqref{k0planewave} or \eqref{laplacepw} are
treated using the same nodes and weights.
Because we have separated the near and far fields, we will
be using these rules under restrictive conditions on $x,y$. 
As we shall see below in more detail, for the doubly periodic case
we will typically invoke the quadrature
under suitable rescaling so that 
$x \in [1,7]$ and $y \in [-2,2]$.
Finding optimal weights and nodes for this restricted range of arguments
leads to a nonlinear optimization problem which can be solved by what is 
known as generalized Gaussian quadrature \cite{ggq1,ggq2,ggq3}.

For the modified Hemholtz
equation, if $\beta$ is bounded away from zero, the integral converges and the 
number of nodes depends rather weakly on $\beta$ itself.
We note, however, that
in the limit $\beta=0$, the modified Helmholtz
integral (for $x>0$) becomes 
\[ 
\int_{0}^{\infty}\frac{e^{-\lambda x}}{\lambda}
[ M_1(\lambda) e^{i\lambda y} + {M_1(-\lambda)} 
e^{-i\lambda y} ] d\lambda, \quad x>0,
\]
with $M_1(\lambda) = O(1/\lambda)$.
Thus, significant adjustments would be required 
as $\beta \rightarrow 0$ to handle the near hypersingularity at the origin.
In the present context, where we seek to impose periodicity, 
charge neutrality is a natural condition (see \cref{sec:lap}).

For the modified Helmholtz equation, special purpose 
quadratures have been constructed for
$\beta$ in different ranges. The number of quadrature nodes decreases as $\beta$
increases. For $\beta \in [10^{-6}, +\infty)$, $x \in [1,7]$, $y \in [-2,2]$, 
at most 21 nodes have been found to yield six digits of accuracy 
and at most 41 nodes have been found to 
yield twelve digits of accuracy. When $\beta > 22$, 
only 1 node is sufficient for six digits of accuracy. When $\beta >37$,
3 nodes are sufficient for twelve digits of accuracy.
For the Poisson equation, with $x \in [1,7]$, $y \in [-2,2]$, 
18 nodes yield six digits of accuracy and 29 nodes yield twelve digits of accuracy. 
We omit consideration of the 
modified Helmholtz equation when 
$\beta < 10^{-6}$ but charge neutrality is not satisfied, as this is a highly
ill-conditioned problem. 
Assuming charge neutrality, one may simply use the quadrauture
designed for the Poisson equation.

\subsection{The non-uniform fast Fourier transform}

For high aspect ratio unit cells, we will require 
the evaluation of
discrete Fourier transforms where the nodes,
frequencies, or both are not uniformly spaced. By
combining the standard fast Fourier transform (FFT)
with careful analysis and fast interpolation
techniques, these sums can be computed with nearly optimal
computational complexity. The resulting algorithms are
known as non-uniform fast Fourier transforms (NUFFTs).
They were originally described in~\cite{nufft2,nufft3}.
We refer the reader to
\cite{finufftlib,finufft} for recent references and a state-of-the-art
implementation.

The type-I NUFFT evaluates sums of the form
\be
\label{eq:nufft1}
f_k = 
\sum_{j=1}^N c_j e^{ikx_j} \; , \textrm{ for } k = -M,\ldots, M.
\ee
Letting $\f = (f_{-M},\dots,f_{M-1},f_M)$ and
letting $\bc = (c_1,\dots,c_N)$, we will write
\[ \f  = \NUFFTone \bc. \]
When the explicit dependence on the
point locations $\{ x_j \}$ and the number of Fourier modes $M$ 
are needed,
we will denote the operator $\NUFFTone$ by
$\NUFFTone(\{x_j\};M)$.
The operator
$\NUFFTone$ can be applied using $O(M\log M+N\log(1/\epsilon))$
operations with nearly the same performance as the standard FFT.

Given the vector $\f$, 
The type-II NUFFT evaluates sums of the form
\be
\label{eq:nufft2}
v_j = 
\sum_{k=-M}^M f_k e^{-ix_jk} \; , \textrm{ for } j = 1,\ldots,
N \; ,
\ee
corresponding to the adjoint of $\NUFFTone$:
\[ \bv = \NUFFTtwo \f , \]
with the same computational complexity,
where $\bv = (v_1,\dots,v_N)$.

\subsection{Legendre polynomials and barycentric interpolation}
\label{sec:interpolation}
The standard Legendre polynomials can be defined by
setting $P_0 \equiv 1$ and $P_1(t) = t$,  with higher
degree polynomials defined by the recurrence formula

\begin{equation}
  \nonumber
  (l+1) P_{l+1} (t) = (2l+1)tP_l(t) - l P_{l-1}(t) \; .
\end{equation}
Let $-1<t_1<\cdots < t_M < 1$ be the roots
of $P_M$, known as the Legendre nodes of order $M$.

Letting $f$ be a function defined on $[-1,1]$, the
degree $M-1$ polynomial, $p_M[f]$, which interpolates $f$
at the Legendre nodes of order $M$, can be written in the form
\begin{equation}
  p_M[f](t) = \frac{\sum_{i=1}^M \frac{\sigma_i}{t-t_i}f(t_i)}
  {\sum_{i=1}^M \frac{\sigma_i}{t-t_i}} \; ,
\end{equation}
where 
\begin{equation}
  \label{eq:baryweights}
  \sigma_i = \frac{1}{\prod_{j\ne i} \left (t_i-t_j \right )} \; .
\end{equation}
This is known as the second form of the barycentric
formula for the interpolant. 

As observed in \cite{wang2012convergence}, if
$f$ is analytic in the Bernstein ellipse with foci
at $\pm 1$ and semi-major and semi-minor lengths adding up
to $\rho > 1$, then

\begin{equation}
  \label{eq:interpbound}
  \|f-p_M[f]\|_\infty \leq (1+\Lambda_M) \frac{2C}{\rho^M(\rho-1)}
  \; ,
\end{equation}
where $\|\cdot \|_\infty$ is the maximum norm on $[-1,1]$,
$C$ is a constant so that $|f| \leq C$ on the Bernstein ellipse,
and $\Lambda_M = O(\sqrt{M})$ is the Lebesgue constant
for the nodes. Thus, the interpolant is a spectrally accurate
approximation of $f$. See \cite{fornberg1998practical,hesthaven2007spectral,trefethen2008gauss,wang2012convergence}
for further details.
\section{Periodicity for the modified Helmholtz equation}
\label{sec:yukawaterms}

In this section, we consider the imposition of 
periodic boundary conditions for the two-dimensional 
modified Helmholtz equation with either one or two directions of
periodicity.
This requires an efficient scheme for the evaluation of 
the field due to all image sources in the far field (outside the nearest 
neighbors of $\C$). For simplicity, we fix $m_0=1$ when considering 
periodicity in the the $\eh_1=(d,0)$ direction alone and $m_0 = 3$ when
considering periodicty in both the 
$\eh_1=(d,0)$ and $\eh_2=(\xi,\eta)$ directions.
Since the governing Green's function is exponentially decaying, all of the 
infinite series in the definition of the periodizing operators
in \eqref{peropdef}
converge absolutely.

Our algorithm is based on 
splitting the far field kernels into
two parts for singly periodic case,
\be\label{ksplitting1}
\ba
K_1^{west}(\bt,\bs) &=
\sum_{m=-\infty}^{-2}
G(\bt,\bs+ (md,0)), \\
K_1^{east}(\bt,\bs) &=
\sum_{m=2}^\infty
G(\bt,\bs+ (md,0)),
\ea
\ee 
and four parts (as in \cref{fig3}) for the doubly periodic case:
\be\label{ksplitting2}
\ba
K_2^{west}(\bt,\bs) &=
\sum_{n=-1}^{1} \sum_{m=-\infty}^{-4}
G(\bt,\bs+\bl_{mn}), \\
K_2^{east}(\bt,\bs) &=
\sum_{n=-1}^{1} \sum_{m=4}^\infty
G(\bt,\bs+\bl_{mn}), \\
K_2^{south}(\bt,\bs) &=
\sum_{n=-\infty}^{-2}\sum_{m=-\infty}^\infty
G(\bt,\bs+\bl_{mn}), \\
K_2^{north}(\bt,\bs) &=
\sum_{n=2}^\infty \sum_{m=-\infty}^\infty
G(\bt,\bs+\bl_{mn}).
\ea
\ee 

The corresponding operators will be denoted by 
$\P_1^{west}$, $\P_1^{east}$,
$\P_2^{west}$, $\P_2^{east}$,
$\P_2^{south}$ and $\P_2^{north}$, so that
\be\label{splitting}
\ba
\P_1 &= \P_1^{west} + \P_1^{east}, \\
\P_2 &= \P_2^{west} + \P_2^{east} + 
\P_2^{south} + \P_2^{north}.  
\ea
\ee

\begin{theorem} \label{south_thm}
Let $\bS = \{ \bs_j \, | \, j = 1,\dots,N_S \}$ and 
$\bT = \{ \bt_l \, | \, l = 1,\dots,N_T \}$ 
denote collections of sources and targets
in the unit cell $\C$ and let $\P_2^{south}$ denote the 
$N_T \times N_S$ operator with
$\P_2^{south} (l,j) = K_2^{south}(\bt_l,\bs_j)$. 
Given a precision $\epsilon$, let
\be\label{mpk0s}
M = 
\left \lceil
\frac{d}{2\pi\eta}\log\left(
\frac{1}{1-e^{-\frac{2\pi\eta}{d}}}\frac{1}{\epsilon}\right)
\right \rceil \approx
\frac{A}{2\pi} (\log(A) + \log(1/\epsilon)).
\ee
For $m=-M,\dots,M$, let
\be \label{chialphadef}
\alpha_m=\frac{2\pi m}{d}, \quad \chi_m=\sqrt{\alpha_m^2+\beta^2}, \quad
Q_m=\chi_m \eta-i\alpha_m\xi \, .
\ee
Let 
$\bL^{south} \in \mathbb{C}^{N_T \times (2M+1)}$ and 
$\bR^{south} \in \mathbb{C}^{(2M+1) \times N_S}$ be dense matrices
and let $\bD^{south} \in \mathbb{C}^{(2M+1) \times (2M+1)}$ be a diagonal
matrix with
\be\label{lowranksouth3}
\ba
\bL^{south}(l,m) &=
e^{-\chi_m y_l} e^{i \alpha_m x_l} \, , \\
\bR^{south}(m,j) &=  
e^{\chi_m y'_j} \, e^{- i \alpha_m x'_j} \, , \\
\bD^{south}(m,m) &= \frac{1}{2 d \,  \chi_m} 
\frac{e^{-2Q_{m}}}{1 - e^{-Q_{m}}}.
\ea
\ee
Then
\begin{equation}  \label{lowrank2d}
\P_2^{south} = \bL^{south} \, \bD^{south} \, \bR^{south} + O(\epsilon). 
\end{equation}
\end{theorem}

\begin{proof}
Combining \eqref{k0planewave} and \eqref{ksplitting2}, we obtain
\be
\ba
K_2^{south}(\bt,\bs_j) 
&=\sum_{n=-\infty}^{-2}\sum_{m=-\infty}^\infty
\int_{-\infty}^\infty
\frac{e^{-\sqrt{\lambda^2+\beta^2}(y-y'_j-n\eta)}}{4 \pi \sqrt{\lambda^2+\beta^2}}
  \, e^{i\lambda (x-x'_j-md-n\xi)} d\lambda\\
\hspace{-0.5in} &= \sum_{n=-\infty}^{-2}\sum_{m=-\infty}^\infty
\int_{-\infty}^\infty
\frac{e^{-\sqrt{\lambda^2+\beta^2}(y-y'_j-n\eta)}}{\sqrt{\lambda^2+\beta^2}}
  \, e^{i\lambda (x-x'_j-n\xi)} \frac{1}{2d}\delta\left(\lambda-\frac{2\pi m}{d}\right) d\lambda \\
\hspace{-0.5in} &= \sum_{n=-\infty}^{-2}\sum_{m=-\infty}^\infty
\frac{1}{2d} 
 \frac{e^{-\chi_m(y-y'_j-n\eta)}}{\chi_m}
\, e^{i\alpha_m(x-x'_j-n\xi)}, 
\ea
\ee
where $\chi_m, \alpha_m$ are given by
\cref{chialphadef}.
The last two equalities follow from the Poisson summation formula 
\cref{deltarep} and the Dirac delta function property.
Exchanging the order of summation and summing the relevant geometric series
in $n$ leads to
\be\label{pk0s}
K_2^{south}(\bt,\bs_j) 
=\frac{1}{2d}
\sum_{m=-\infty}^\infty
\frac{e^{-\chi_m(y-y'_j)+i\alpha_m(x-x'_j)}}{\chi_m} 
\frac{e^{-2Q_m}}{1-e^{-Q_m}}
\ee
where $Q_m$ is given in \cref{chialphadef}.

Truncating the sum at $|m| = M$ yields the $(2M+1)$ term approximation
\be\label{tpk0s}
{K_{2,M}}^{south}(\bt,\bs_j) =
\frac{1}{2d}
\sum_{m=-M}^{M}
\frac{e^{-\chi_m(y-y'_j)+i\alpha_m(x-x'_j)}}{\chi_m} 
\frac{e^{-2Q_m}}{1-e^{-Q_m}}
\ee
and the formulas in \cref{lowranksouth3}.

We now estimate the truncation error
\be
E_s= |{K_2}^{south}(\bt,\bs_j) - {K_{2.M}}^{south}(\bt,\bs_j) |.
\ee
For any $m\in \mathbb{Z}$, it is easy to verify that
\[ \chi_m\ge \frac{2\pi |m|}{d}, \ 
\left|e^{-\chi_m(y-y_j)-Q_m}\right|=
\left|e^{-\chi_m(\eta+y-y_j)}\right|\le 1,
\]
\[
\left|e^{-Q_m}\right|\le e^{-\chi_m\eta}, \ 
\left|1-e^{-Q_m}\right|\ge 1-e^{-\chi_m\eta}.
\]
From these bounds, it follows that
\be\label{esbound1}
\ba
E_s&\le \frac{1}{d}\sum_{m=M+1}^{\infty}\frac{1}{\chi_m}
\frac{e^{-\chi_m\eta}}{1-e^{-\chi_m\eta}}
< \frac{1}{M+1}\frac{e^{-\frac{2\pi (M+1)\eta}{d}}}{1-e^{-\frac{2\pi (M+1) \eta}{d}}}
\sum_{k=0}^{\infty}
e^{-\frac{2\pi k\eta}{d}}\\
&=\frac{1}{2 \pi (M+1)}
\frac{e^{-\frac{2\pi (M+1)\eta}{d}}}{\left(1-e^{-\frac{2\pi (M+1) \eta}{d}}\right)
  \left(1-e^{-\frac{2\pi\eta}{d}}\right)}.
\ea
\ee
with $M$ given by \cref{mpk0s}.
The estimate \cref{lowrank2d} follows, completing the proof.
\end{proof}

\begin{remark}
\cref{south_thm} yields a truncated version of
the formula \cref{lowranksouth} with
weights $\bc=(c_{-M},c_{-M+1},\dots,c_{M-1},c_{M})$ 
given by 
\[
\bc = \bD^{south} \, \bR^{south} \bq \, ,
\]
where $\bq=(q_1,\dots,q_{N_S})$ is the vector of charge strengths.
\end{remark}

\begin{remark} 
When $\beta \gg 1$, the value of $M$ in \cref{mpk0s} can be shown
to be even smaller, but since the cost is logarithmic in $A$ and $\epsilon$
we omit this more detailed analysis.
\end{remark}

\begin{remark}
The observation that Poisson summation yields rapidly converging
series approximations for lines or half spaces of lattice points
that do not pass through the origin was made in
\cite{enoch2001,mcphedran2000jmp} for the purpose of computing
lattice sums. 
\end{remark}

Essentially the same analysis yields
\begin{corollary}
The matrix $\P_2^{north}$ 
has the low-rank factorization
\[
\P_2^{north} = \bL^{north} \, \bD^{north} \, \bR^{north} + O(\epsilon)
\]
where $\bD^{north} = \overline{\bD^{south}}$, 
\be\label{lowranknorth}
\ba
\bL^{north}(l,m) &=
e^{\chi_m y_l} e^{i \alpha_m x_l} \, , \\
\bR^{north}(m,j) &=  
e^{-\chi_m y'_j} \, e^{- i \alpha_m x'_j} \, .
\ea
\ee
\end{corollary}

\begin{definition}  \label{epsrankdef}
Since the rank is precision-dependent, we say that
$\P_2^{north}$ has an 
$\epsilon$-rank of $2M+1$.
\end{definition}

It remains to consider the ``west" and ``east" contributions.
\begin{theorem}
Let $\bt,\bs_j$ lie in a unit cell $\C$ 
and $m_0 = 1$. Then, the kernels
$K_1^{west}$ and $K_1^{east}$ have the integral representations
\be\label{pylw0}
\ba
K_1^{west}(\bt,\bs_j) 
&= \int_{-\infty}^\infty
\frac{e^{-\sqrt{\lambda^2+\beta^2}(x-x'_j)}}{4 \pi \sqrt{\lambda^2+\beta^2}}
\,  e^{i\lambda (y-y'_j)} 
\frac{e^{-2\sqrt{\lambda^2+\beta^2}d}}
     {1-e^{-\sqrt{\lambda^2+\beta^2}d}}
     d\lambda, \\
K_1^{east}(\bt,\bs_j) 
&= \int_{-\infty}^\infty
\frac{e^{\sqrt{\lambda^2+\beta^2}(x-x'_j)}}{4 \pi \sqrt{\lambda^2+\beta^2}}
\,  e^{i\lambda (y-y'_j)} 
\frac{e^{-2\sqrt{\lambda^2+\beta^2}d}}
     {1-e^{-\sqrt{\lambda^2+\beta^2}d}}
     d\lambda. 
\ea
\ee
Letting $m_0 = 3$ with the additional assumption with $d \geq \sqrt{\xi^2+\eta^2}$, the kernels 
$K_2^{west}$ and $K_2^{east}$ 
have the integral representations 
\be\label{pylw02}
\ba
K_2^{west}(\bt,\bs_j) 
&= \sum_{n=-1}^1\int_{-\infty}^\infty
\frac{e^{-\sqrt{\lambda^2+\beta^2}(x-x'_j-n\xi)}}{4 \pi \sqrt{\lambda^2+\beta^2}}
\,  e^{i\lambda (y-y'_j-n\eta)} 
\frac{e^{-4\sqrt{\lambda^2+\beta^2}d}}
     {1-e^{-\sqrt{\lambda^2+\beta^2}d}}
     d\lambda, \\
K_2^{east}(\bt,\bs_j) 
&= \sum_{n=-1}^1\int_{-\infty}^\infty
\frac{e^{\sqrt{\lambda^2+\beta^2}(x-x'_j-n\xi)}}{4 \pi \sqrt{\lambda^2+\beta^2}}
\,  e^{i\lambda (y-y'_j-n\eta)} 
\frac{e^{-4\sqrt{\lambda^2+\beta^2}d}}
     {1-e^{-\sqrt{\lambda^2+\beta^2}d}}
     d\lambda.
\ea
\ee
\end{theorem}

\begin{proof}
These formulas follow directly from
\cref{k0planewave} and \cref{ksplitting1} and summation of the geometric
series in $m$. In the singly periodic case, excluding one nearest neighbor from
either side is sufficient to ensure the exponential decay of the integrand
in \eqref{pylw0}.
In the doubly periodic case, with a parallelogram as the unit cell,
we must ensure that we are using the integral representation
of the modified Bessel function where it is valid and that the resulting integrand
decays exponentially fast.
For this, we must have that
$x-x'-md-n\xi>d$ for all $m\le -(m_0+1)$ and $|n|\leq 1$ in the ``west"
case, and
$x-x'+md-n\xi>d$ for all $m\ge m_0+1$ and $|n|\leq 1$ in the ``east" case.
It is straightforward to verify that if $m_0=3$, then
$x-x'-md-n\xi\ge 4d-(d+\xi)-\xi>d$ in the first instance and that 
$x-x'+md-n\xi>d$ in the second instance 
under the stated assumption
about the unit cell. 
\end{proof}

The reader will note that there is a major difference between the east/west
representations and those for the north and south. 
The latter are fully discrete,
while for the east and west representations, we have an integral that needs
to be evaluated before we can develop a low-rank decomposition. 
For this, we will make use of numerical
quadrature, in order to develop a low-rank
approximation of precision $\epsilon$.
This provides a discrete approximation of the Sommerfeld representation
for $K_1^{west}(\bt,\bs_j)$ and $K_2^{west}(\bt,\bs_j)$
in \cref{pylw0,pylw02} rewritten in the form
\be
\label{tk0west}
\ba
K_1^{west}(\bt,\bs_j) 
&=  \Re \left( \int_{0}^\infty
\frac{e^{-\sqrt{\lambda^2+\beta^2}(x-x'_j)}}{2 \pi \sqrt{\lambda^2+\beta^2}}
\,  e^{i\lambda (y-y'_j)} 
\frac{e^{-2\sqrt{\lambda^2+\beta^2}d}}
     {1-e^{-\sqrt{\lambda^2+\beta^2}d}}
     d\lambda  \right) \\
&\hspace{-0.25in} \approx \Re \Bigg(
\sum_{n=1}^{N^1_q(\beta,d,\eta)}  w_{n,1}(\beta,d,\eta)
\frac{e^{-\sqrt{\lambda_{n,1}^2+\beta^2}(x-x'_j)}}
{2 \pi \sqrt{\lambda_{n,1}^2+\beta^2}}
\,  e^{i\lambda_{n,1} (y-y'_j)} 
\frac{e^{-2\sqrt{\lambda_{n,1}^2+\beta^2}d}}
     {1-e^{-\sqrt{\lambda_{n,1}^2+\beta^2}d}} \Bigg). \\
K_2^{west}(\bt,\bs_j) 
&=  \Re \Bigg( \int_{0}^\infty
\frac{e^{-\sqrt{\lambda^2+\beta^2}(x-x'_j)}}{2 \pi \sqrt{\lambda^2+\beta^2}}
\,  e^{i\lambda (y-y'_j)} 
\frac{e^{-4\sqrt{\lambda^2+\beta^2}d}}
     {1-e^{-\sqrt{\lambda^2+\beta^2}d}}  \\
&\hspace{1in} 
[ e^{-\sqrt{\lambda^2+\beta^2} \xi + i \lambda \eta} +
  e^{\sqrt{\lambda^2+\beta^2} \xi - i \lambda \eta} + 1] \,
     d\lambda  \Bigg) \\
&\hspace{-0.25in} \approx \Re \Bigg(
\sum_{n=1}^{N^2_q(\beta,d,\eta)}  w_{n,2}(\beta,d,\eta)
\frac{e^{-\sqrt{\lambda_{n,2}^2+\beta^2}(x-x'_j)}}{2 \pi \sqrt{\lambda_{n,2}^2+\beta^2}}
\,  e^{i\lambda_{n,2}(y-y'_j)} 
\frac{e^{-4\sqrt{\lambda_{n,2}^2+\beta^2}d}}
     {1-e^{-\sqrt{\lambda_{n,2}^2+\beta^2}d}}  \\
&\hspace{1in} 
[ e^{-\sqrt{\lambda_{n,2}^2+\beta^2} \xi + i \lambda_{n,2} \eta} +
  e^{\sqrt{\lambda_{n,2}^2+\beta^2} \xi - i \lambda_{n,2} \eta} + 1] \, 
\Bigg).
\ea
\ee
Note that different numbers of nodes 
may be needed for the two cases.
We denote by $N^1_q$ and $N^2_q$ the number of nodes needed
for $K_1^{west}$ and 
$K_2^{west}$, respectively, with weights and nodes
$\{ w_{n,1}, \lambda_{n,1} \}$ and 
$\{ w_{n,2}, \lambda_{n,2} \}$.

In both cases, we define $\lambda'=\lambda/d$ so that
the decaying exponential in the integrand decays at least as fast as
$e^{-\sqrt{\lambda'^2+\tilde{\beta}^2}}$ ($\tilde{\beta}=\beta\cdot d$). 
In the doubly periodic case, this leads to the consideration of the integrals
in \cref{sec:ggq} where $(x,y)\in [1,7]\times[-2,2]$. Generalized Gaussian
quadrature can be applied to construct numerical quadratures for
a given precision $\epsilon$, with a weak dependence on $\beta$.
These quadratures are valid for unit cells with arbitrary 
geometric parameters and thus can be precomputed and stored
\cite{ggq1,ggq2,ggq3}.

In the singly periodic case, $x$ lies in $[1,3]$, while the range of $y$
can be very large when $\eta\gg d$, leading to highly oscillatory integrals.
In this case, the quadrature is constructed as follows. First, 
the interval $[0,\infty)$ is truncated to $[0,L]$, which can be
accomplished easily due to the exponential decay in the $x$ variable,
with $L=\sqrt{(\log(1/\epsilon))^2-\beta^2}$.
If $\log(1/\epsilon)\le\beta$, we can set $L=0$, since
the whole integral is then negligible. 
Second, in order to accurately capture the oscillatory
behavior in the $y$ variable, the interval $[0,L]$ is further divided
into subintervals $[j\lambda_0,(j+1)\lambda_0]$ for $j=0,\ldots, \lceil L/\lambda_0\rceil$,
where $\lambda_0=2\pi d/\eta$. A shifted and scaled $n$ point 
Gauss-Legendre quadrature rule (with $n=O(\log(1/\epsilon))$)
is then applied to discretize the integral on each subinterval
$[j\lambda_0,(j+1)\lambda_0]$ for $j\ge 1$. 
Third, when $\beta$ is very small, a new difficulty emerges - namely that
the integrand is nearly singular at the origin. 
In that case, we further divide $[0,\lambda_0]$
into dyadic subintervals $[0,a]$ and $[2^{k-1}a,2^{k}a]$ for $k=1,\ldots,l_{\rm max}$,
where $l_{\rm max}=\lceil\log_2(\lambda_0/\beta)\rceil$ and $a=\lambda_0/l_{\rm max}$.
A shifted and scaled $n$ point Gauss-Legendre quadrature rule is again applied to
discretize the integral on each such subinterval. To summarize, the total number of
quadrature nodes $N_q^1$ (i.e., the numerical rank of the periodizing operator) is
$O\left(\log(1/\epsilon)\left(\log(1/\beta)+\log(1/\epsilon)\cdot \lceil\frac{\eta}{d}\rceil\right)\right)$.
In the limit $\beta \rightarrow 0$, it is also possible to develop asymptotic
expansions in $\beta$, which we do not consider here.

\begin{remark}
  The difference between the singly and doubly periodic cases 
seems rather significant in terms of quadrature design. 
However, this distinction is somewhat artificial. 
The reason that the quadrature problem is simple in 
the doubly periodic case is that 
we have the freedom to choose which lattice vector is oriented along
the $x$-axis. The difficult direction to deal with is the {\em short} axis of the 
unit cell and, by our convention, this makes the 
north/south periodizing kernels more oscillatory which are already discrete.
Thus, the number of
terms in the plane-wave expansion for the north/south parts will
grow linearly with
respect to the aspect ratio $\frac{d}{\eta}$ but without the need for quadrature
design.

To summarize, the numerical rank of the periodizing operators may grow linearly 
with respect to the aspect ratio for both singly and doubly periodic problems. 
When the rank $r$ is large,
the NUFFT can be used to reduce the computational cost from
$O(r (N_T+N_S))$ to $O(\log(1/\epsilon)(r\log r + (N_T+N_S)\log(1/\epsilon)))$
with $\epsilon$ the prescribed precision.
\end{remark}

\begin{theorem} \label{west_thm}
Let $\bS = \{ \bs_j \, | \, j = 1,\dots,N_S \}$ and 
$\bT = \{ \bt_l \, | \, l = 1,\dots,N_T \}$ 
denote collections of sources and targets
in the unit cell $\C$ and let $\P_1^{west}, \P_2^{west}$ denote the 
$N_T \times N_S$ operators with
$\P_1^{west} (l,j) = K_1^{west}(\bt_l,\bs_j)$ and
$\P_2^{west} (l,j) = K_2^{west}(\bt_l,\bs_j)$. 
Given a precision $\epsilon$, let $N^1_q(\beta,d,\eta),
N^2_q(\beta,d,\eta)$ denote the 
number of points needed in the numerical quadrature 
for $K_1^{west}(\bt,\bs)$ and $K_2^{west}(\bt,\bs)$, 
with weights and nodes $\{ w_{n,1},\lambda_{n,1} \}$
$\{ w_{n,2},\lambda_{n,2} \}$, respectively.
Let 
$\bL_1^{west} \in \mathbb{C}^{N_T \times N^1_q}$,
$\bL_2^{west} \in \mathbb{C}^{N_T \times N^2_q}$,
$\bR_1^{west} \in \mathbb{C}^{N^1_q \times N_S}$,
$\bR_2^{west} \in \mathbb{C}^{N^2_q \times N_S}$
be dense matrices
and let $\bD_1^{e/w}$, $\bD_2^{west}$
be diagonal matrices of dimension $N^1_q$ and $N^2_q$, respectively,
with
\be\label{lowrankwest}
\ba
\bL_1^{west}(l,n) &=
e^{-\sqrt{\lambda_{n,1}^2 + \beta^2}x_l} e^{i \lambda_{n,1} y_l} \, , \\
\bL_2^{west}(l,n) &=
e^{-\sqrt{\lambda_{n,2}^2 + \beta^2}x_l} e^{i \lambda_{n,2} y_l} \, , \\
\bR_1^{west}(n,j) &=  
e^{\sqrt{\lambda_{n,1}^2+\beta^2} x'_j} \, e^{-i \lambda_{n,1} y'_j} \, , \\
\bR_2^{west}(n,j) &=  
e^{\sqrt{\lambda_{n,2}^2+\beta^2} x'_j} \, e^{-i \lambda_{n,2} y'_j} \, , \\
\bD_1^{e/w}(n,n) &=  \frac{1}{2 \pi} \frac{w_{n,1}}{\sqrt{\lambda_{n,1}^2+\beta^2}} 
\frac{e^{-2\sqrt{\lambda_{n,1}^2+\beta^2}d}}
{1 - e^{-\sqrt{\lambda_{n,1}^2+\beta^2}d}} \, , \\
\bD_2^{west}(n,n) &= \frac{w_{n,2}}{\sqrt{\lambda_{n,2}^2+\beta^2}}
\frac{e^{-4\sqrt{\lambda_{n,2}^2+\beta^2}d}}
{1 - e^{-\sqrt{\lambda_{n,2}^2+\beta^2}d}}
\frac{[ e^{-\sqrt{\lambda_{n,2}^2+\beta^2} \xi + i \lambda_{n,2} \eta} +
  e^{\sqrt{\lambda_{n,2}^2+\beta^2} \xi - i \lambda_{n,2} \eta} + 1]}{2 \pi}.
\ea
\ee
Let
\[
\ba
\P_1^{west} &= \bL_1^{west} \, \bD_1^{e/w} \, \bR_1^{west}, \\
\P_2^{west} &= \bL_2^{west} \, \bD_2^{west} \, \bR_2^{west}. \\
\ea
\]
Then the real parts of the vectors
\[  \P_1^{west} \bq, \ \P_2^{west} \bq 
\]
denote the contributions from the west sources to the 
corresponding periodizing potentials.
\end{theorem}

\begin{corollary}
The matrices $\P_1^{east}$ and $\P_2^{east}$
have the low-rank factorizations
\[
\ba
\P_1^{east} &= \bL_1^{east} \, \bD_1^{e/w} \, \bR_1^{east} + O(\epsilon) \\
\P_2^{east} &= \bL_2^{east} \, \bD_2^{east} \, \bR_2^{east} + O(\epsilon) 
\ea
\]
where $\bD_2^{east} = \overline{\bD_2^{west}}$, 
\be\label{lowrankeast}
\ba
\bL_1^{east}(l,n) &=
e^{\sqrt{\lambda_{n,1}^2 + \beta^2}x_l} e^{i \lambda_{n,1} y_l} \, , \\
\bL_2^{east}(l,n) &=
e^{\sqrt{\lambda_{n,2}^2 + \beta^2}x_l} e^{i \lambda_{n,2} y_l} \, , \\
\bR_1^{east}(n,j) &=  
e^{-\sqrt{\lambda_{n,1}^2+\beta^2} x'_j} \, e^{-i \lambda_{n,1} y'_j} \, ,\\
\bR_2^{east}(n,j) &=  
e^{-\sqrt{\lambda_{n,2}^2+\beta^2} x'_j} \, e^{-i \lambda_{n,2} y'_j} \, .
\ea
\ee
\end{corollary}


%
%
%
\section{Periodizing operators for the Poisson equation} \label{sec:lap}

In this section, we derive formulas for
$\P_1^{west}$, $\P_1^{east}$,
$\P_2^{west}$, $\P_2^{east}$,
$\P_2^{south}$ and $\P_2^{north}$ in the limit $\beta \rightarrow 0$,
allowing us to impose periodic boundary conditions for the Poisson equation
using the same formalism
\be\label{splitting_lap}
\ba
\P_1 &= \P_1^{west} + \P_1^{east}, \\
\P_2 &= \P_2^{west} + \P_2^{east} + 
\P_2^{south} + \P_2^{north}.  
\ea
\ee

As noted earlier, we require charge neutrality for the periodic
problem to be well-posed. Moreover, as is well-known, 
the potential is only unique up to an arbitrary constant.

\begin{theorem} \label{south_thm_lap}
Let $\bS = \{ \bs_j \, | \, j = 1,\dots,N_S \}$,
$\bq=(q_1,\dots,q_{N_S})$, and
$\bT = \{ \bt_l \, | \, l = 1,\dots,N_T \}$ 
denote collections of source locations, charge strengths and targets
in the unit cell $\C$ 
with $\sum_{j=1}^{N_S} q_j = 0$. 
Given a precision $\epsilon$, let 
\be\label{mpk0s_lap}
M = 
\left \lceil
\frac{d}{2\pi\eta}\log\left(
\frac{1}{1-e^{-\frac{2\pi\eta}{d}}}\frac{1}{\epsilon}\right)
\right \rceil \approx
\frac{A}{2\pi} (\log(A) + \log(1/\epsilon)).
\ee
For $m=-M.\dots,M$, let
\be \label{alphaqdef}
\alpha_m=\frac{2\pi m}{d}, \quad
Q_m=\alpha_m \eta-i\alpha_m\xi \, .
\ee
Let 
$\bL^{south} \in \mathbb{C}^{N_T \times (2M+1)}$ and 
$\bR^{south} \in \mathbb{C}^{(2M+1) \times N_S}$ be dense matrices
and let $\bD^{south} \in \mathbb{C}^{(2M+1) \times (2M+1)}$ be a diagonal
matrix with
\be\label{lowranksouth3_lap}
\begin{split}
&\bL^{south}(l,m) \ \ =
e^{-|\alpha_m| y_l} e^{i \alpha_m x_l} \quad {\rm for}\ m \neq 0, \\
&\bL^{south}(l,0) \ \ \ =  
y_l, \\
&\bR^{south}(m,j)\  =  
e^{|\alpha_m| y'_j} \, e^{- i \alpha_m x'_j} \quad {\rm for}\ m \neq 0, \\
&\bR^{south}(0,j) \ \ =   
y'_j, \\ 
&\bD^{south}(m,m) = 
\frac{1}{4 \pi |m|} 
\frac{e^{-2Q_{m}}}{1 - e^{-Q_{m}}}  \quad {\rm for}\ m \neq 0 \\
&\bD^{south}(0,0) \ \  =  
-\frac{1}{2 d\eta}. 
\end{split}
\ee
Then
\[
\P_2^{south} = \bL^{south} \, \bD^{south} \, \bR^{south} + O(\epsilon). 
\]
\end{theorem}

\begin{proof}
For all modes $m \neq 0$, this result follows directly from taking the 
limit $\beta \rightarrow 0$ in the corresponding term for 
the modified Helmholtz equation. 
For $m=0$, the relevant contribution to 
$K_2^{south}(\bt,\bs_j)$ in \cref{pk0s}  is
\[ \frac{1}{2d} \frac{1}{\beta} 
\frac{e^{-2\beta \eta}}{1-e^{-\beta\eta}}
e^{-\beta(y-y'_j)}.
\] 
Letting $u_{south}$ denote the field due to all sources in the ``south"
image cells and
summing over all sources yields
\be\label{phis0}
u_{south}(\bt) =\frac{1}{2d} 
\frac{1}{\beta} \sum_{j=1}^{N_S} 
\frac{e^{-2\beta \eta}}{1-e^{-\beta\eta}}
e^{-\beta(y-y'_j)} q_j.
\ee
Differentiating both sides of \eqref{phis0} with respect to $y$, we have
\be\label{phis1}
\frac{\partial u_{south}(x,y)}{\partial y}
=
-\frac{1}{2d}
\frac{e^{-2\beta \eta}}{1-e^{-\beta \eta}}
\sum_{j=1}^{N_S}e^{-\beta(y-y'_j)}q_j .
\ee
Taylor expansion of the various terms yields:
\[
\frac{\partial u_{south}(x,y)}{\partial y}
= - \frac{1}{2d}  \, 
\frac{ 1 - O(\beta)}{\beta \eta + O(\beta^2)} \, 
e^{-\beta y} \, \left[ 
\sum_{j=1}^{N_S} q_j + \beta 
\sum_{j=1}^{N_S} q_j y_j' \right]
\]
Taking the limit $\beta\rightarrow 0$ and using charge neutrality
\cref{chargeneutral},
we obtain
\be\label{phis2}
\lim_{\beta\rightarrow 0}\frac{\partial u_{south}(x,y)}{\partial y}
=-\frac{1}{2 d\eta}\sum_{j=1}^{N_S}y_j q_j.
\ee
Hence, $u_{south}(x,y)$ is given by
\be\label{phis3}
\lim_{\beta\rightarrow 0} u_{south}(x,y)
=-\frac{1}{2 d\eta}\left(\sum_{j=1}^{N_S}y_j q_j\right) y
\ee
up to an arbitrary constant, completing the derivation.
\end{proof}
It is easy to verify the following.

\begin{corollary} \label{north_thm_lap}
Under the hypotheses of \cref{south_thm_lap},
let 
$\bL^{north} \in \mathbb{C}^{N_T \times (2M+1)}$ and 
$\bR^{north} \in \mathbb{C}^{(2M+1) \times N_S}$ be dense matrices
with
\be\label{lowranknorth3_lap}
\begin{split}
&\bL^{north}(l,m) \ \ =
e^{|\alpha_m| y_l} e^{i \alpha_m x_l} \quad {\rm for}\ m \neq 0, \\
&\bL^{north}(l,0) \ \ \ =  
y_l, \\
&\bR^{north}(m,j)\  =  
e^{-|\alpha_m| y'_j} \, e^{- i \alpha_m x'_j} \quad {\rm for}\ m \neq 0, \\
&\bR^{north}(0,j) \ \ =   
y'_j. 
\end{split}
\ee
Then
\[
\P_2^{north} = \bL^{north} \, \bD^{north} \, \bR^{north} + O(\epsilon),
\]
where $\bD^{north}=\overline{\bD^{south}}$.
\end{corollary}

Similar care needs to be taken when deriving the east and west
formulas in the limit $\beta \rightarrow 0$.

\begin{theorem} \label{west_thm_lap}
Let $\bS = \{ \bs_j \, | \, j = 1,\dots,N_S \}$, $\bq = (q_1,\dots,q_{N_S})$,
and $\bT = \{ \bt_l \, | \, l = 1,\dots,N_T \}$ 
denote collections of source locations, charge strengths, and targets
in the unit cell $\C$, with $\sum_{j=1}^{N_S} q_j = 0$.
Given a precision $\epsilon$, let $N^1_q(d,\eta),
N^2_q(d,\eta)$ denote the 
number of points needed in the numerical quadrature 
for the $K_1^{west}$ and $K_2^{west}$ kernels (see \cref{k1lap}
and \cref{k2lap} below),
with weights and nodes $\{ w_{n,1},\lambda_{n,1} \}$
$\{ w_{n,2},\lambda_{n,2} \}$, respectively. 
Let 
$\bL_1^{west}, \bL_1^{east} \in \mathbb{C}^{N_T \times N^1_q}$,
$\bL_2^{west}, \bL_2^{east} \in \mathbb{C}^{N_T \times N^2_q}$,
$\bR_1^{west}, \bR_1^{east} \in \mathbb{C}^{N^1_q \times N_S}$,
$\bR_2^{west}, \bR_2^{east} \in \mathbb{C}^{N^2_q \times N_S}$
be dense matrices
and let $\bD_1^{e/w}$, $\bD_2^{west}$ and $\bD_2^{east}$ 
be diagonal matrices of dimension $N^1_q$ and $N^2_q$, respectively,
with
\be\label{lowrankwest_lap}
\ba
\bL_1^{west}(l,n) &=
e^{-\lambda_{n,1} x_l} e^{i \lambda_{n,1} y_l} \, , \ 
\bL_2^{west}(l,n) =
e^{-\lambda_{n,2} x_l} e^{i \lambda_{n,2} y_l} \, , \\
\bL_1^{east}(l,n) &=
e^{\lambda_{n,1} x_l} e^{i \lambda_{n,1} y_l} \, , \ 
\bL_2^{east}(l,n) =
e^{\lambda_{n,2} x_l} e^{i \lambda_{n,2} y_l} \, , \\
\bR_1^{west}(n,j) &=  
e^{\lambda_{n,1} x'_j} \, e^{-i \lambda_{n,1} y'_j} \, , \ 
\bR_2^{west}(n,j) =  
e^{\lambda_{n,2} x'_j} \, e^{-i \lambda_{n,2} y'_j} \, , \\
\bR_1^{east}(n,j) &=  
e^{-\lambda_{n,1} x'_j} \, e^{-i \lambda_{n,1} y'_j} \, , \ 
\bR_2^{east}(n,j) =  
e^{-\lambda_{n,2} x'_j} \, e^{-i \lambda_{n,2} y'_j} \, , \\
\bD_1^{e/w}(n,n) &= \frac{1}{2\pi} \frac{w_{n,1}}{\lambda_{n,1}}
\frac{e^{-2\lambda_{n,1} d}}
{1 - e^{-\lambda_{n,1} d}} \, , \\
\bD_2^{west}(n,n) &= \frac{1}{2\pi}  \frac{w_{n,2}}{\lambda_{n,2}} 
\frac{e^{-4\lambda_{n,2} d}}
{1 - e^{-\lambda_{n,2} d}} 
[ e^{-\lambda_{n,2} \xi + i \lambda_{n,2} \eta} +
  e^{\lambda_{n,2} \xi - i \lambda_{n,2} \eta} + 1],
\ea
\ee
and $\bD_2^{east}=\overline{\bD_2^{west}}$.
Let
\[
\ba
\P_1^{west} &= \bL_1^{west} \, \bD_1^{e/w} \, \bR_1^{west}, \\
\P_2^{west} &= \bL_2^{west} \, \bD_2^{west} \, \bR_2^{west}. \\
\P_1^{east} &= \bL_1^{east} \, \bD_1^{e/w} \, \bR_1^{east}, \\
\P_2^{east} &= \bL_2^{east} \, \bD_2^{east} \, \bR_2^{east}. \\
\ea
\]
Then the real parts of the vectors 
\[  \P_1^{west} \bq, \ \P_1^{east} \bq, \ 
  P_2^{west} \bq, \P_2^{east} \bq 
\]
denote the contributions from the west or east sources to the 
corresponding periodizing potentials.
\end{theorem}

\begin{proof}
Focusing on the ``west" sources, 
the formulas themselves follow directly from the modified 
Helmholtz case, letting $\beta \rightarrow 0$ in
\cref{tk0west} and applying generalized Gaussian quadrature.
As noted in \cref{sec:ggq}, however, the quadrature rule is now
being used to evaluate an integral of the apparent form
\be
\label{k1lap}
\int_{0}^{\infty}\frac{e^{-\lambda x} e^{i \lambda y}}{\lambda}
M_1(\lambda) 
\frac{e^{-2\lambda d}}{1-e^{-\lambda d}}
d\lambda
\ee
for the singly periodic case or
\be
\label{k2lap}
\int_{0}^{\infty}\frac{e^{-\lambda x} e^{i \lambda y}}{\lambda}
M_1(\lambda) 
\frac{e^{-4\lambda d}}{1-e^{-\lambda d}}
[ e^{-\lambda \xi + i \lambda \eta} +
  e^{\lambda \xi - i \lambda \eta} + 1]
d\lambda
\ee
for the doubly periodic case.
Since $\frac{e^{-2\lambda d}}{1-e^{-\lambda d}}$ and 
$\frac{e^{-4\lambda d}}{1-e^{-\lambda d}}$ are both of the order
$O(\frac{1}{\lambda})$ as $\lambda\rightarrow 0$, the integrals appear to be strongly singular,
with a $\frac{1}{\lambda^2}$ singularity at the origin.
However, in applying the periodizing operator, we are limiting ourselves
to charge neutral distributions, so that
\[ M_1(\lambda) = \sum_{j=1}^{N_S} q_j e^{\lambda (x_j - i y_j)}
= O(\lambda).
\]
Moreover, in computing any physical
quantity, such as the gradient of the potential, a second factor of
$\lambda$ is introduced in the numerator and thus, the generalized
Gaussian quadrature rule is only being applied to integrals of the form
\[ 
\int_{0}^{\infty}e^{-\lambda x} e^{i \lambda y}
M_{2}(\lambda) \, d\lambda,
\]
where $M_{2}(\lambda)$ is smooth.
The analysis for the ``east" sources is identical. 
\end{proof}

\section{Periodizing operator for the modified Stokes equations}
\label{sec:stokes}

The modified Stokeslet is the fundamental solution to the
modified Stokes equations
\be
\ba
(\beta^2-\Delta)\bu + \nabla p&=0\\
\nabla \cdot \bu&=0
\ea
\ee
and is given at $\bt = (x,y)$ by
\be
\G^{(\rm MS)}(\bt)=(\nabla \otimes \nabla-\Delta\bI)G_{\rm MB}(\bt) =
\begin{pmatrix}
-\partial_{yy} & \partial_{xy}  \\
\partial_{xy} & -\partial_{xx}  
\end{pmatrix}
G_{\rm MB}(\bt)
\label{msgreen}
\ee
where 
\be
G_{\rm MB}(\bt)=-\frac{1}{2\pi \beta^2}
\left[K_0(\beta |\bt|)- \log(1/|\bt|)\right].
\ee
This is the fundamental solution for the modified biharmonic equation:
\be
\Delta(\Delta-\beta^2)G_{\rm MB}(\bt)=
\delta(\bt).
\ee
Taking the Fourier transform of both sides yields
the representation
\be
G_{\rm MB}(\bt)=\frac{1}{4\pi^2}\int_{-\infty}^\infty\int_{-\infty}^\infty
\frac{1}{(k_1^2+k^2_2)(\beta^2+k_1^2+k^2_2)}
e^{i(k_1 x+k_2y)}dk_1dk_2.
\label{mbfourier}
\ee
Substituting \eqref{mbfourier} into \eqref{msgreen}, we obtain
the Fourier representation of the modified Stokeslet:
\be
\G^{(\rm MS)}_{ij}(\bt)=\frac{1}{4\pi^2}\int_{-\infty}^\infty\int_{-\infty}^\infty
\frac{(k_1^2+k_2^2)\delta_{ij}-k_ik_j}{(k_1^2+k^2_2)(\beta^2+k_1^2+k^2_2)}
e^{i(k_1 x+k_2y)}dk_1dk_2.
\ee
We now extend Sommerfeld's method to derive a plane-wave
expansion for the modified Stokeslet
(valid for $x>0$) by contour integration in the $k_1$
variable and the residue theorem.

For this, note that in the complex $k_1$-plane, the integrand has
four poles, namely $\pm i\sqrt{\beta^2+k_2^2}$ and $\pm i |k_2|$.
Under the assumption that $x>0$, 
consider the closed contour from
$-R$ to $R$ along the real $k_1$ axis and returning along a 
semicircle of radius $R$
in the upper half of the complex plane. The integral
along the semicircle clearly vanishes
as $R\rightarrow \infty$, since 
\[ |e^{i k_1 x}|\le 1 \, , \quad
\frac{1}{\beta^2+R^2e^{2i\theta}+k_2^2}\rightarrow 0 \quad 
{\rm as}\ R\rightarrow \infty
\]
and the remaining terms in the integrand are bounded by $2$.
From the residue theorem, it follows that 
the integral is due to the residues at the two poles
$i\sqrt{\beta^2+k_2^2}$ and $i|k_2|$ that lie within the contour, 
leading to:
\be
\ba
\G^{(\rm MS)}(\bt)&=\frac{1}{4\pi\beta^2}\int_{-\infty}^\infty
\frac{e^{-\sqrt{\beta^2+\lambda^2}x}}{\sqrt{\beta^2+\lambda^2}} e^{i\lambda y}
\begin{bmatrix}-\lambda^2 & i\lambda\sqrt{\beta^2+\lambda^2}\\
i\lambda\sqrt{\beta^2+\lambda^2}& \beta^2+\lambda^2\end{bmatrix}d\lambda\\
&+\frac{1}{4\pi\beta^2}\int_{-\infty}^\infty e^{-|\lambda|x+i\lambda y}
\begin{bmatrix} |\lambda| & -i\lambda \\
-i\lambda & -|\lambda| \end{bmatrix} d\lambda, \qquad x>0.
\ea
\label{mspwexpeast}    
\ee
The plane-wave expansions for $x<0, y>0$, and $y < 0$ 
are obtained similarly. Here, we have renamed the $k_2$ Fourier variable 
as $\lambda$ to be consistent with our earlier notation.

\subsection{Low rank factorization}

The periodizing operators $\P_1$ and $\P_2$ can be constructed
by the same method as for the modified Helmholtz equation:
\be\label{splitting_ms}
\ba
\P_1 &= \P_1^{west} + \P_1^{east}, \\
\P_2 &= \P_2^{west} + \P_2^{east} + 
\P_2^{south} + \P_2^{north}.  
\ea
\ee

For a source to the ``south", we have
\be\label{pmss0}
\ba
\K_2^{south}(\bt,\bs_j)&=\sum_{n=-\infty}^{-2}\sum_{m=-\infty}^\infty
\G^{(\rm MS)}(\bt-(\bs+\bl_{mn}))\\
&=\frac{1}{4\pi\beta^2}\sum_{n=-\infty}^{-2}\sum_{m=-\infty}^\infty
\left\{\int_{-\infty}^\infty
\frac{e^{-\sqrt{\beta^2+\lambda^2}(y-y'_j-n\eta)}}{\sqrt{\beta^2+\lambda^2}}
e^{i\lambda(x-x'_j-md-n\xi)}\right.\\
& \qquad \qquad \qquad \qquad \qquad \cdot
\begin{bmatrix} \beta^2+\lambda^2 & i\lambda\sqrt{\beta^2+\lambda^2}\\
  i\lambda\sqrt{\beta^2+\lambda^2} & -\lambda^2
\end{bmatrix}
d\lambda\\
&\,\,\,\,\,\,\left.+\int_{-\infty}^\infty 
e^{-|\lambda|(y-y'_j-n\eta)+i\lambda(x-x'_j-md-n\xi)}
\begin{bmatrix} -|\lambda| & -i\lambda\\ -i\lambda& |\lambda| \end{bmatrix}
d\lambda\right\}.
\ea
\ee

Following the same procedure used for the modified Helmholtz
equation above, we obtain 
\be\label{pmss1}
\ba
\K_2^{south}(\bt,\bs_j)&=\frac{1}{2d\beta^2}\sum_{m=-\infty}^\infty\left(
\begin{bmatrix} \chi_m & i\alpha_m\\ i\alpha_m & -\frac{\alpha_m^2}{\chi_m} \end{bmatrix}
\frac{e^{-2Q^{(\beta)}_m}}{1-e^{-Q^{(\beta)}_m}} e^{-\chi_m(y-y'_j)+i\alpha_m(x-x'_j)}\right.\\
&\qquad\qquad\quad \left.
+\begin{bmatrix} -|\alpha_m| & -i\alpha_m\\ -i\alpha_m  & |\alpha_m| \end{bmatrix}
\frac{e^{-2Q_m}}{1-e^{-Q_m}} e^{-|\alpha_m|(y-y'_j)+i\alpha_m(x-x'_j)}\right),
\ea
\ee
where
\be \label{chialphadefms}
 Q^{(\beta)}_m=\chi_m\eta-i\alpha_m\xi, \quad 
Q_m=|\alpha_m|\eta-i\alpha_m\xi,
\ee
with $\alpha_m$, $\chi_m$ given in \cref{chialphadef}.
This establishes 
\begin{theorem} \label{south_thm_ms}
Let $\bS = \{ \bs_j \, | \, j = 1,\dots,N_S \}$ and 
$\bT = \{ \bt_l \, | \, l = 1,\dots,N_T \}$ 
denote collections of sources and targets
in the unit cell $\C$ and let $\P_2^{south}$ denote the 
$N_T \times N_S$ block matrix with
$\P_2^{south} (l,j) = \K_2^{south}(\bt_l,\bs_j)$. 
Given a precision $\epsilon$, let $M$ be given by
\cref{mpk0s}.
For $m=-M.\dots,M$, let
$Q^{(\beta)}_m, Q_m$ be given by \cref{chialphadefms} and let
$\alpha_m, \chi_m$ be given by \cref{chialphadef}.
Let 
$\bL_\beta^{south}, \bL^{south}, \bL_\beta^{north}, \bL^{north} \in \mathbb{C}^{2N_T \times 2(2M+1)}$ and 
$\bR_\beta^{south}, \bR^{south}, \bR_\beta^{north}, \bR^{north} \in \mathbb{C}^{2(2M+1) \times 2N_S}$ 
be dense $N_T \times (2M+1)$ 
and $(2M+1) \times N_S$ 
block matrices, respectively, with $2\times 2$ blocks,
let $\bD_\beta^{south}, \bD^{south}, \bD_\beta^{north}, \bD^{north} \in \mathbb{C}^{2(2M+1) \times 2(2M+1)}$
be $(2M+1)\times (2M+1)$ block diagonal matrices with $2\times 2$ diagonal
blocks, let $\bI_2$ denote the identity matrix of size 2, and let
\be\label{lowranksouth3ms}
\ba
\bL_\beta^{south}(l,m) &=
e^{-\chi_m y_l} e^{i \alpha_m x_l} \bI_2 , \quad 
\bL_\beta^{north}(l,m) =
e^{\chi_m y_l} e^{i \alpha_m x_l} \bI_2 \, , \\ 
\bL^{south}(l,m) &=
e^{-|\alpha_m| y_l} e^{i \alpha_m x_l} \bI_2 , \quad
\bL^{north}(l,m) =
e^{|\alpha_m| y_l} e^{i \alpha_m x_l} \bI_2 \, , \\
\bR_\beta^{south}(m,j) &=
e^{\chi_m y'_j} e^{-i \alpha_m x'_j} \bI_2 , \quad
\bR_\beta^{north}(m,j) =
e^{-\chi_m y'_j} e^{-i \alpha_m x'_j} \bI_2 \, , \\
\bR^{south}(m,j) &=
e^{|\alpha_m| y'_j} e^{-i \alpha_m x'_j} \bI_2, \quad
\bR^{north}(m,j) =
e^{-|\alpha_m| y'_j} e^{-i \alpha_m x'_j} \bI_2 \, , \\
\bD_\beta^{south}(m,m) &=
\frac{1}{2\beta^2 d} 
\frac{e^{-2Q^{(\beta)}_m}}{1 - e^{-Q^{(\beta)}_m}} \, 
\begin{pmatrix}
\chi_m & i \alpha_m \\
i \alpha_m & -\alpha_m^2/\chi_m 
\end{pmatrix} \, , \quad \bD_\beta^{north}(m,m) =\overline{\bD_\beta^{south}(m,m)}\, ,\\
\bD^{south}(m,m) &=
\frac{1}{2\beta^2 d} 
\frac{e^{-2Q_m}}{1 - e^{-Q_m}} \, 
\begin{pmatrix}
-|\alpha_m| & -i \alpha_m \\
-i \alpha_m & |\alpha_m| 
\end{pmatrix} \, , \quad \bD^{north}(m,m) =\overline{\bD^{south}(m,m)}\, .
\ea
\ee
Then
\[
\ba
\P_2^{south} &= \bL_\beta^{south} \, \bD_\beta^{south} \, \bR_\beta^{south} +
 \bL^{south} \, \bD^{south} \, \bR^{south} + O(\epsilon), \\
\P_2^{north} &= \bL_\beta^{north} \, \bD_\beta^{north} \, \bR_\beta^{north} +
 \bL^{north} \, \bD^{north} \, \bR^{north} + O(\epsilon). 
\ea
\]
\end{theorem}

For the sources in image boxes to the ``west", we have
\be\label{pmsw0}
\ba
\K_1^{west}(\bt,\bs_j)&=\sum_{m=-\infty}^{-2}
\G^{(\rm MS)}\left(\bt - (\bs_j+ \bl_{m0})\right)\\
&=\frac{1}{4\pi\beta^2}
\left\{\int_{-\infty}^\infty e^{-\sqrt{\beta^2+\lambda^2}(x-x'_j)}e^{i\lambda(y-y'_j)}\right.\\
&\qquad \qquad \qquad \cdot
\begin{bmatrix}-\frac{\lambda^2}{\sqrt{\beta^2+\lambda^2}} & i\lambda\\
  i\lambda& \sqrt{\beta^2+\lambda^2}\end{bmatrix}
\frac{e^{-2\sqrt{\beta^2+\lambda^2}d}}{1-e^{-\sqrt{\beta^2+\lambda^2}d}}  d\lambda\\
&\qquad\qquad\qquad+\int_{-\infty}^\infty e^{-|\lambda|(x-x'_j)+i\lambda(y-y'_j)}\\
&\left.\qquad\qquad\qquad\cdot\begin{bmatrix} |\lambda| & -i\lambda \\
-i\lambda & -|\lambda| \end{bmatrix}
\frac{e^{-2|\lambda|d}}{1-e^{-|\lambda|d}}  d\lambda\right\}, \\
\K_2^{west}(\bt,\bs_j)&=\sum_{m=-\infty}^{-4}\sum_{n=-1}^1
\G^{(\rm MS)}\left(\bt - (\bs_j+ \bl_{mn})\right)\\
&=\frac{1}{4\pi\beta^2}\sum_{n=-1}^1
\left\{\int_{-\infty}^\infty e^{-\sqrt{\beta^2+\lambda^2}(x-x'_j-n\xi)}e^{i\lambda(y-y'_j-n\eta)}\right.\\
&\qquad \qquad \qquad \cdot
\begin{bmatrix}-\frac{\lambda^2}{\sqrt{\beta^2+\lambda^2}} & i\lambda\\
  i\lambda& \sqrt{\beta^2+\lambda^2}\end{bmatrix}
\frac{e^{-4\sqrt{\beta^2+\lambda^2}d}}{1-e^{-\sqrt{\beta^2+\lambda^2}d}}  d\lambda\\
&\qquad\qquad\qquad+\int_{-\infty}^\infty e^{-|\lambda|(x-x'_j-n\xi)+i\lambda(y-y'_j-n\eta)}\\
&\left.\qquad\qquad\qquad\cdot\begin{bmatrix} |\lambda| & -i\lambda \\
-i\lambda & -|\lambda| \end{bmatrix}
\frac{e^{-4|\lambda|d}}{1-e^{-|\lambda|d}}  d\lambda\right\}.
\ea
\ee
\begin{theorem} \label{west_thm_ms}
Let $\bS = \{ \bs_j \, | \, j = 1,\dots,N_S \}$ and 
$\bT = \{ \bt_l \, | \, l = 1,\dots,N_T \}$ 
denote collections of sources and targets
in the unit cell $\C$ and let $\P_1^{west}, \P_2^{west}$ denote the 
$N_T \times N_S$ block matrices with $2 \times 2$ blocks 
$\P_1^{west} (l,j) = \K_1^{west}(\bt_l,\bs_j)$ and
$\P_2^{west} (l,j) = K_2^{west}(\bt_l,\bs_j)$. 
Given a precision $\epsilon$, let $N^1_q(\beta,d,\eta),N^1_q(0,d,\eta)$
and $N^2_q(\beta,d,\eta), N^2_q(0,d,\eta)$ 
denote the 
number of points needed in the numerical quadratures
for the two integrals in each of 
$\K_1^{west}(\bt,\bs)$ and $\K_2^{west}(\bt,\bs)$, 
with weights and nodes $\{ w_{n,\beta,1},\lambda_{n,\beta,1} \}$,
$\{ w_{n,0,1},\lambda_{n,0,1} \}$,
$\{ w_{n,\beta,2},\lambda_{n,\beta,2} \}$, and
$\{ w_{n,0,2},\lambda_{n,0,2} \}$, respectively.
Let 
$\bL_{1,\beta}^{west}$,
$\bL_1^{west}$,
$\bL_{2,\beta}^{west}$,
$\bL_2^{west}$,
$\bR_{1,\beta}^{west}$,
$\bR_1^{west}$,
$\bR_{2,\beta}^{west}$,
$\bR_2^{west}$
be dense block matrices with $2\times 2$ blocks given by:
\be\label{lowrankwest_ms1}
\ba
\bL_{1,\beta}^{west}(l,n) &=
e^{-\sqrt{\lambda_{n,\beta,1}^2 + \beta^2}x_l} e^{i \lambda_{n,\beta,1} y_l} 
\bI_2 \, |\quad l = 1,\dots,N_T, \  n = 1,\dots,N^1_q(\beta,d,\eta) \\
\bL_1^{west}(l,n) &=
e^{-|\lambda_{n,0,1}|x_l} e^{i \lambda_{n,0,1} y_l} 
\bI_2 \, \qquad |\quad l = 1,\dots,N_T, \  n = 1,\dots,N^1_q(0,d,\eta) \\
\bL_{2,\beta}^{west}(l,n) &=
e^{-\sqrt{\lambda_{n,\beta,2}^2 + \beta^2}x_l} e^{i \lambda_{n,\beta,2} y_l} 
\bI_2 \, |\quad l = 1,\dots,N_T, \  n = 1,\dots,N^2_q(\beta,d,\eta) \\
\bL_2^{west}(l,n) &=
e^{-|\lambda_{n,0,2}|x_l} e^{i \lambda_{n,0,2} y_l} 
\bI_2 \, \qquad |\quad l = 1,\dots,N_T, \  n = 1,\dots,N^2_q(0,d,\eta) \\
\bR_{1,\beta}^{west}(n,j) &=
e^{\sqrt{\lambda_{n,\beta,1}^2 + \beta^2}x'_j} e^{-i \lambda_{n,\beta,1} y'_j} 
\bI_2 \, |\quad n = 1,\dots,N^1_q(\beta,d,\eta), \  j = 1,\dots,N_S \\
\bR_1^{west}(n,j) &=
e^{|\lambda_{n,0,1}|x'_j} e^{-i \lambda_{n,0,1} y'_j} 
\bI_2 \, \qquad |\quad n = 1,\dots,N^1_q(0,d,\eta), \  j = 1,\dots,N_S \\
\bR_{2,\beta}^{west}(n,j) &=
e^{\sqrt{\lambda_{n,\beta,2}^2 + \beta^2}x'_j} e^{-i \lambda_{n,\beta,2} y'_j} 
\bI_2 \, |\quad n = 1,\dots,N^2_q(\beta,d,\eta), \  j = 1,\dots,N_S \\
\bR_2^{west}(n,j) &=
e^{|\lambda_{n,0,2}|x'_j} e^{i \lambda_{n,0,2} y'_j} 
\bI_2 \ \ \qquad |\quad n = 1,\dots,N^2_q(0,d,\eta), \  j = 1,\dots,N_S \, ,
\ea
\ee
and let
$\bD_{1,\beta}^{west}$,
$\bD_{1}^{west}$,
$\bD_{2,\beta}^{west}$,
$\bD_{2}^{west}$
be block diagonal matrices 
with $2\times 2$ blocks given by:
%
\be\label{lowrankwest_ms2}
\ba
\bD_{1,\beta}^{west}(n,n) &=  \frac{w_{n,\beta,1}}{4 \pi \beta^2} 
\frac{e^{-2\sqrt{\lambda_{n,\beta,1}^2+\beta^2}d}}
{1 - e^{-\sqrt{\lambda_{n,\beta,1}^2+\beta^2}d}} \, 
\begin{bmatrix}-\frac{\lambda_{n,1}^2}{\sqrt{\beta^2+\lambda_{n,\beta,1}^2}}
  & i\lambda_{n,\beta,1}\\
  i\lambda_{n,\beta,1}& \sqrt{\beta^2+\lambda_{n,\beta,1}^2}\end{bmatrix}, \\
&\hspace{2.8in} \  n = 1,\dots, N^1_q(\beta,d,\eta) \\
\bD_1^{west}(n,n) &=  \frac{w_{n,0,1}}{4 \pi \beta^2} 
\frac{e^{-2|\lambda_{n,0,1}|d}}
{1 - e^{-|{\lambda_{n,0,1}}|d}} \, 
\begin{bmatrix} |\lambda_{n,0,1}|& -i \lambda_{n,0,1} \\
-i \lambda_{n,0,1}  & -|\lambda_{n,0,1}| 
\end{bmatrix}, \\
&\hspace{2.8in} \  n = 1,\dots, N^1_q(0,d,\eta) \\
\bD_{2,\beta}^{west}(n,n) &=  \frac{w_{n,\beta,2}}{4 \pi \beta^2} 
[ e^{-\sqrt{\beta^2 + \lambda_{n,\beta,2}} \xi + i \lambda_{n,\beta,2} \eta} +
  e^{\lambda_{n,\beta,2} \xi - i \lambda_{n,\beta,2} \eta} + 1] \cdot \\
&\hspace{0.5in}
\frac{e^{-4\sqrt{\lambda_{n,\beta,2}^2+\beta^2}d}}
{1 - e^{-\sqrt{\lambda_{n,\beta,2}^2+\beta^2}d}} \, 
\begin{bmatrix}-\frac{\lambda_{n,\beta,2}^2}{\sqrt{\beta^2+\lambda_{n,\beta,2}^2}}
  & i\lambda_{n,\beta,2}\\
  i\lambda_{n,\beta,2}& \sqrt{\beta^2+\lambda_{n,\beta,2}^2}\end{bmatrix}, \\
&\hspace{2.8in}  \  n = 1,\dots, N^2_q(\beta,d,\eta) \\
\bD_2^{west}(n,n) &=  \frac{w_{n,0,2}}{4 \pi \beta^2} 
[ e^{-|\lambda_{n,0,2}| \xi + i \lambda_{n,0,2} \eta} +
  e^{|\lambda_{n,0.2}| \xi - i \lambda_{n,0.2} \eta} + 1] \cdot \\
&\hspace{0.5in}
\frac{e^{-4|\lambda_{n,0,2}|d}}
{1 - e^{-|\lambda_{n,0,2}|d}} \, 
\begin{bmatrix}|\lambda_{n,0,1}| & -i\lambda_{n,0,1}\\ 
-i\lambda_{n,0,1}& - |\lambda_{n,0,1}|\end{bmatrix},
\quad  n = 1,\dots, N^2_q(0,d,\eta) .
\ea
\ee
Let
\[
\ba
\P_1^{west} &= \bL_{1,\beta}^{west} \, \bD_{1,\beta}^{west} \, \bR_{1,\beta}^{west} +
 \bL_1^{west} \, \bD_{1}^{west} \, \bR_1^{west}, \\
\P_2^{west} &= \bL_{2,\beta}^{west} \, \bD_{2,\beta}^{west} \, \bR_{2,\beta}^{west} +
 \bL_2^{west} \, \bD_{2}^{west} \, \bR_2^{west}. \\
\ea
\]
Then the real parts of the vectors
\[  \P_1^{west} \bq, \ \P_2^{west} \bq 
\]
denote the contributions from the west sources to the 
corresponding periodizing potentials.
The formulas for 
$\P_1^{east}$ and $\P_2^{east}$ are identical, except that 
$x_l \leftrightarrow -x_l$ and 
$x'_j \leftrightarrow -x'_j$  in the various 
$\bL(l,n)$ and $\bR(n,j)$ blocks above and that
$\bD_{1,\beta}^{east}=\overline{\bD_{1,\beta}^{west}}$,
$\bD_{1}^{east}=\overline{\bD_{1}^{west}}$,
$\bD_{2,\beta}^{east}=\overline{\bD_{2,\beta}^{west}}$,
$\bD_{2}^{east}=\overline{\bD_{2}^{west}}$.
\end{theorem}

\subsection{Periodizing operator for the Stokes equations}
While the Stokeslet, i.e., the Green's function for the incompressible Stokes flow,
is given by the formula
\be
\G^{(\rm S)}(\bt)=-\frac{1}{4\pi}\left(\log |\bt|\,\bI-\frac{\bt\otimes \bt}{|\bt|^2}\right),
\label{stokeslet}
\ee
a systematic way of computing the correct limit for the periodizing 
operators is to let 
$\beta\rightarrow 0$ in the various formulas for the modified Stokes
equations, invoking charge neutrality before taking the limit.

\begin{theorem} \label{south_thm_s}
Let $\bS = \{ \bs_j \, | \, j = 1,\dots,N_S \}$ and 
$\bT = \{ \bt_l \, | \, l = 1,\dots,N_T \}$ 
denote collections of sources and targets
in the unit cell $\C$ and let $\P_2^{south}$ denote the 
$N_T \times N_S$ block matrix which is the periodizing
operator for all ``south" sources.
Given a precision $\epsilon$, let $M$ be given by
\cref{mpk0s}.
With $\alpha_m, Q_m$ given in \cref{chialphadef,chialphadefms},
let 
$\bL^{south}$ and  
$\bR^{south}$ be defined as in \cref{lowranksouth3ms}
except with 
\be\label{southmzero}
\ba
\bL^{south}(l,0) &= 
\begin{pmatrix}
y_l & 0  \\
0  & y_l
\end{pmatrix} \\
\bR^{south}(0,j) &= 
\begin{pmatrix}
y'_j & 0  \\
0  & y'_j
\end{pmatrix}.
\ea
\ee
Let $\bD_a^{south}, \bD_b^{south}
 \in \mathbb{C}^{2(2M+1) \times 2(2M+1)}$ be 
$(2M+1)\times (2M+1)$ 
block diagonal matrices with $2\times 2$ diagonal
blocks, and let
$\bD_S
 \in \mathbb{C}^{2N_S \times 2N_S}$,
$\bD_T
 \in \mathbb{C}^{2N_T \times 2N_T}$ 
be block diagonal matrices with $2\times 2$ diagonal blocks
given by 
\be\label{lowranksouth3s}
\ba
\bD_a^{south}(m,m) &=
\frac{1}{4d} 
\frac{e^{-2Q_m}}{1 - e^{-Q_m}} \, 
\left( \frac{1}{|\alpha_m|} 
\begin{pmatrix}
1 & 0 \\
0 & 1 \end{pmatrix} \, -
\frac{2-e^{-Q_m}}{1 - e^{-Q_m}} \eta \, 
\begin{pmatrix}
1 & i \sign(m) \\
i \sign(m) & -1
\end{pmatrix}
\right)  \, \\
&\hspace{2in} \quad {\rm for}\ m \neq 0 \\
\bD_b^{south}(m,m) &=
\frac{1}{4d} 
\frac{e^{-2Q_m}}{1 - e^{-Q_m}} \, 
\begin{pmatrix}
1 & i \sign(m) \\
i \sign(m) & -1
\end{pmatrix} \,  \quad {\rm for}\ m \neq 0 \\
\bD_a^{south}(0,0) &=
-\frac{1}{2d \eta} 
\begin{pmatrix}
1 & 0 \\
0 & 0 \end{pmatrix} \, , \quad
\bD_b^{south}(0,0) =
\begin{pmatrix}
0 & 0 \\
0 & 0 \end{pmatrix} 
\, \\
\bD_S(j,j) &=
\begin{pmatrix}
y'_j & 0 \\
0 & y'_j
\end{pmatrix} \, \\
\bD_T(i,i) &=
\begin{pmatrix}
y_i & 0 \\
0 & y_i
\end{pmatrix} \, .
\ea
\ee
Then
\[
\ba
\P_2^{south} &= \bL^{south} \, \bD_a^{south} \, \bR^{south} -
 \bD_T \, \bL^{south} \, \bD_b^{south} \, \bR^{south} +
 \bL^{south} \, \bD_b^{south} \, \bR^{south} \, \bD_S  +
 O(\epsilon)\, .
\ea
\]
\end{theorem}

\begin{proof}
Consider first one of the terms in \cref{pmss1} corresponding to a mode
$m \neq 0$. We will denote the limit as $\beta\to 0$ by $\K_2^{south}[m]$.
Using L'Hopital's rule, and taking the limit $\beta \rightarrow 0$,
it is straightforward to see that
\be\label{pssm}
\ba
\K_2^{south}[m](\bt,\bs_j) 
&= \lim_{\beta \rightarrow 0}
\frac{1}{2d\beta^2}\left(
\begin{bmatrix} \chi_m
  & i\alpha_m \\ i\alpha_m & -\frac{\alpha_m^2}{\chi_m} 
\end{bmatrix}
\frac{e^{-2Q^{(\beta)}_m}}{1-e^{-Q^{(\beta)}_m}} e^{-\chi_m(y-y'_j)+i\alpha_m(x-x'_j)}\right. \\
&\left.\qquad\qquad\qquad
+\begin{bmatrix} -|\alpha_m| & -i\alpha_m \\ -i\alpha_m & |\alpha_m| \end{bmatrix}
\frac{e^{-2Q_m}}{1-e^{-Q_m}} e^{-|\alpha_m|(y-y'_j)+i\alpha_m(x-x'_j)}\right)\\
&=\frac{1}{4d}\left\{
\frac{1}{|\alpha_m|}\begin{bmatrix} 1 & 0 \\ 0 & 1 \end{bmatrix}
-\left(y-y_0+\frac{2-e^{-Q_m}}{1-e^{-Q_m}}\eta\right)
\begin{bmatrix} 1 & i\sign(m) \\ i\sign(m) & -1 \end{bmatrix}\right\}\\
&\qquad\quad\cdot \frac{e^{-2Q_m}}{1-e^{-Q_m}}e^{-|\alpha_m|(y-y'_j)+i\alpha_m(x-x'_j)}.
\ea
\ee
It is easy to check that every column of $\K_2^{south}[m]$ is divergence-free 
and that every entry of $\K_2^{south}[m]$ is biharmonic.

For the $m=0$ term, we have
\be
\K_2^{south}[0](\bt,\bs_j) = \lim_{\beta\to 0}\frac{1}{2d\beta^2}
\begin{bmatrix} \beta & 0\\ 0 & 0 \end{bmatrix}
\frac{e^{-2\beta \eta}}{1-e^{-\beta \eta}} e^{-\beta(y-y'_j)}.
\ee
As we did for the Poisson equation, using charge neutrality and 
expanding the exponential terms in a Taylor series, we obtain
\be\label{pss0}
\K_2^{south}[0](\bt,\bs_j) 
=-\frac{1}{2d\eta} \begin{bmatrix} y y'_j & 0\\ 0 & 0 \end{bmatrix}.
\ee
Combining \eqref{pssm} and \eqref{pss0}, we obtain 
\be
\ba
\K_2^{south}(\bt,\bs_j)&= \sum_{m=-\infty}^\infty \K_2^{south}[m](\bt,\bs_j) \\
&=-\frac{1}{2d\eta} \begin{bmatrix} y y'_j & 0\\ 0 & 0 \end{bmatrix}
+\frac{1}{4d} \sum_{\substack{m=-\infty \\ m\ne 0}}^\infty\left\{
\frac{1}{|\alpha_m|}\begin{bmatrix} 1 & 0 \\ 0 & 1 \end{bmatrix}\right.\\
&\left.-\left(y-y'_j+\frac{2-e^{-Q_m}}{1-e^{-Q_m}}\eta\right)
\begin{bmatrix} 1 & i\sign(m) \\ i\sign(m) & -1 \end{bmatrix}\right\}\\
&\qquad\qquad\qquad\cdot \frac{e^{-2Q_m}}{1-e^{-Q_m}}
e^{-|\alpha_m|(y-y'_j)+i\alpha_m(x-x'_j)}.
\ea
\ee

\end{proof}

\begin{remark}
In an almost identical manner, we can show that
\be
\ba
\K_2^{north}(\bt,\bs_j)
&=-\frac{1}{2d\eta} \begin{bmatrix} y y'_j & 0\\ 0 & 0 \end{bmatrix}
+\frac{1}{4d} \sum_{\substack{m=-\infty \\ m\ne 0}}^\infty\left\{
\frac{1}{|\alpha_m|}\begin{bmatrix} 1 & 0 \\ 0 & 1 \end{bmatrix}\right.\\
&\left.-\left(y-y'_j-\frac{2-e^{-\overline{Q_m}}}{1-e^{-\overline{Q_m}}}\eta\right)
\begin{bmatrix} -1 & i\sign(m) \\ i\sign(m) & 1 \end{bmatrix}\right\}\\
&\qquad\qquad\qquad\cdot \frac{e^{-2\overline{Q_m}}}{1-e^{-\overline{Q_m}}}
e^{|\alpha_m|(y-y'_j)+i\alpha_m(x-x'_j)}.
\ea
\ee
And the expression for $\P_2^{north}$ can be derived similarly.  
\end{remark}
Taking the limit $\beta\rightarrow 0$ for \eqref{pmsw0} and using
charge neutrality, we likewise 
obtain the west part of the periodizing operator for the Stokeslet:
\be\label{pssw0}
\ba
\K_1^{west}(\bt,\bs)&=
\frac{1}{8\pi} \int_{-\infty}^\infty
\frac{e^{-2|\lambda|d}}{1-e^{-|\lambda|d}} e^{-|\lambda|(x-x')}e^{i\lambda(y-y')}
\cdot\left\{\frac{1}{|\lambda|}\begin{bmatrix} 1 & 0 \\ 0 & 1 \end{bmatrix}\right.\\
&\left.  -\left(x-x'+\frac{2-e^{-|\lambda|d}}{1-e^{-|\lambda|d}}d\right)
\begin{bmatrix}-1 & i\sign(\lambda)\\ i\sign(\lambda)& 1\end{bmatrix}\right\}
d\lambda, \\
\K_2^{west}(\bt,\bs)&=
\frac{1}{8\pi}\sum_{n=-1}^1 \int_{-\infty}^\infty
\frac{e^{-4|\lambda|d}}{1-e^{-|\lambda|d}} e^{-|\lambda|(x-x'-n\xi)}e^{i\lambda(y-y'-n\eta)}
\cdot\left\{\frac{1}{|\lambda|}\begin{bmatrix} 1 & 0 \\ 0 & 1 \end{bmatrix}\right.\\
&\left.  -\left(x-x'-n\xi+\frac{4-3e^{-|\lambda|d}}{1-e^{-|\lambda|d}}d\right)
\begin{bmatrix}-1 & i\sign(\lambda)\\ i\sign(\lambda)& 1\end{bmatrix}\right\}
d\lambda,
\ea
\ee
It is again easy to check that every column of $\K_1^{west}$ or 
$\K_2^{west}$ is divergence-free 
and that every entry is biharmonic. The above representation
yields the following theorem.

\begin{theorem} \label{west_thm_s}
Let $\bS = \{ \bs_j \, | \, j = 1,\dots,N_S \}$ and 
$\bT = \{ \bt_l \, | \, l = 1,\dots,N_T \}$ 
denote collections of sources and targets
in the unit cell $\C$ and let $\P_1^{west}, \P_2^{west}$ denote the 
$N_T \times N_S$ block matrices with $2 \times 2$ blocks 
$\P_1^{west} (l,j) = \K_1^{west}(\bt_l,\bs_j)$ and
$\P_2^{west} (l,j) = K_2^{west}(\bt_l,\bs_j)$. 
Given a precision $\epsilon$, let $N^1_q(d,\eta)$
and $N^2_q(d,\eta)$ 
denote the 
number of points needed in the numerical quadratures
for the integrals in 
$\K_1^{west}(\bt,\bs)$ and $\K_2^{west}(\bt,\bs)$, 
with weights and nodes 
$\{ w_{n,0,1},\lambda_{n,0,1} \}$,
$\{ w_{n,0,2},\lambda_{n,0,2} \}$, respectively.
Let 
$\bL_1^{west}$,
$\bL_2^{west}$,
$\bR_1^{west}$,
$\bR_2^{west}$
be dense block matrices with $2\times 2$ blocks given by
\cref{lowrankwest_ms1},
and let
$\bD_{1,a}^{west}$,
$\bD_{1,b}^{west}$,
$\bD_{2,a}^{west}$,
$\bD_{2,b}^{west}$
be block diagonal matrices 
with $2\times 2$ blocks given by:
\be\label{lowrankwest_s}
\ba
\bD_{1,a}^{west}(n,n) &=  \frac{w_{n,0,1}}{8 \pi} 
\frac{e^{-2|\lambda_{n,0,1}|d}}
{1 - e^{-|{\lambda_{n,0,1}}|d}} \, 
\Bigg(
\begin{bmatrix} 1/|\lambda_{n,0,1}|& 0\\
0  & 1/|\lambda_{n,0,1}| 
\end{bmatrix}  \\
&\hspace{1in} -
\frac{2 - e^{-|\lambda_{n,0,1}|d}}
{1 - e^{-|{\lambda_{n,0,1}}|d}} \, d \, 
\begin{bmatrix} -1 & i \sign(\lambda_{n,0,1})\\
i \sign(\lambda_{n,0,1})  & 1
\end{bmatrix} \Bigg) , \\
\bD_{1,b}^{west}(n,n) &=  \frac{w_{n,0,1}}{8 \pi} 
\frac{e^{-2|\lambda_{n,0,1}|d}}
{1 - e^{-|{\lambda_{n,0,1}}|d}} \, 
\begin{bmatrix} -1 & i \sign(\lambda_{n,0,1})\\
i \sign(\lambda_{n,0,1})  & 1
\end{bmatrix}, \\
\bD_{2,a}^{west}(n,n) &=  \frac{w_{n,0,2}}{8 \pi} 
\frac{e^{-4|\lambda_{n,0,2}|d}}
{1 - e^{-|{\lambda_{n,0,2}}|d}} \, 
[ e^{-|\lambda_{n,0,2}| \xi + i \lambda_{n,0,2} \eta} +
  e^{|\lambda_{n,0,2}| \xi - i \lambda_{n,0,2} \eta} + 1] \times \\
&
\Bigg(
\begin{bmatrix} 1/|\lambda_{n,0,2}|& 0\\
0  & 1/|\lambda_{n,0,2}| 
\end{bmatrix} -
\frac{4 - 3e^{-|\lambda_{n,0,2}|d}}
{1 - e^{-|{\lambda_{n,0,2}}|d}} \, d \, 
\begin{bmatrix} -1 & i \sign(\lambda_{n,0,2})\\
i \sign(\lambda_{n,0,2})  & 1
\end{bmatrix} \\
&
- [ e^{-|\lambda_{n,0,2}| \xi + i \lambda_{n,0,2} \eta} -
  e^{|\lambda_{n,0,2}| \xi - i \lambda_{n,0,2} \eta} ] \, \xi
\begin{bmatrix} -1 & i \sign(\lambda_{n,0,2})\\
i \sign(\lambda_{n,0,2})  & 1
\end{bmatrix} \Bigg) \\
\bD_{2,b}^{west}(n,n) &=  \frac{w_{n,0,2}}{8 \pi} 
\frac{e^{-4|\lambda_{n,0,2}|d}}
{1 - e^{-|\lambda_{n,0,2}|d}} \, 
[ e^{-|\lambda_{n,0,2}| \xi + i \lambda_{n,0,2} \eta} +
  e^{|\lambda_{n,0,2}| \xi - i \lambda_{n,0,2} \eta} + 1] \times \\
&\hspace{1in}
\begin{bmatrix} -1 & i \sign(\lambda_{n,0,2})\\
i \sign(\lambda_{n,0,2})  & 1
\end{bmatrix} .
\ea
\ee
Let
\[
\ba
\P_1^{west} &= \bL_1^{west} \, \bD_{1,a}^{west} \, \bR_1^{west} 
- \bD_T \bL_1^{west} \bD_{1,b}^{west} \, \bR_1^{west}
+ \bL_1^{west} \bD_{1,b}^{west} \, \bR_1^{west} \, \bD_S \\
\P_2^{west} &= \bL_2^{west} \, \bD_{2,a}^{west} \, \bR_2^{west} 
- \bD_T \bL_2^{west} \bD_{2,b}^{west} \, \bR_2^{west}
+ \bL_2^{west} \bD_{2,b}^{west} \, \bR_2^{west} \, \bD_S \, .
\ea
\]
Then the real parts of the vectors
\[  \P_1^{west} \bq, \ \P_2^{west} \bq 
\]
denote the contributions from the west sources to the 
corresponding periodizing potentials.
\end{theorem}

\begin{remark}
In an almost identical manner, we can show that
\be
\ba
\K_1^{east}(\bt,\bs)&=
\frac{1}{8\pi} \int_{-\infty}^\infty
\frac{e^{-2|\lambda|d}}{1-e^{-|\lambda|d}} e^{|\lambda|(x-x')}e^{i\lambda(y-y')}
\cdot\left\{\frac{1}{|\lambda|}\begin{bmatrix} 1 & 0 \\ 0 & 1 \end{bmatrix}\right.\\
&\left.  -\left(x-x'-\frac{2-e^{-|\lambda|d}}{1-e^{-|\lambda|d}}d\right)
\begin{bmatrix}1 & i\sign(\lambda)\\ i\sign(\lambda)& -1\end{bmatrix}\right\}
d\lambda, \\
\K_2^{east}(\bt,\bs)&=
\frac{1}{8\pi}\sum_{n=-1}^1 \int_{-\infty}^\infty
\frac{e^{-4|\lambda|d}}{1-e^{-|\lambda|d}} e^{|\lambda|(x-x'-n\xi)}e^{i\lambda(y-y'-n\eta)}
\cdot\left\{\frac{1}{|\lambda|}\begin{bmatrix} 1 & 0 \\ 0 & 1 \end{bmatrix}\right.\\
&\left.  -\left(x-x'-n\xi-\frac{4-3e^{-|\lambda|d}}{1-e^{-|\lambda|d}}d\right)
\begin{bmatrix}1 & i\sign(\lambda)\\ i\sign(\lambda)& -1\end{bmatrix}\right\}
d\lambda.
\ea
\ee
And the expressions for $\P_1^{east}$, $\P_2^{east}$ can be derived similarly.
\end{remark}

For both the modified Stokeslet and Stokeslet, the associated pressurelet
is given by:
\be
   {\bf p}(\bt)=\frac{1}{2\pi}\frac{\bt}{|\bt|^2}=\frac{1}{2\pi}\nabla \log|\bt|.
   \label{pressurelet}
\ee
Thus, the periodizing operators for the pressurelet can 
be obtained by simply differentiating those for the 
logarithmic kernel in \cref{sec:lap}, summarized in the following two
theorems.
\begin{theorem} \label{south_thm_stokes_pressure}
Let $\bS = \{ \bs_j \, | \, j = 1,\dots,N_S \}$ and 
$\bT = \{ \bt_l \, | \, l = 1,\dots,N_T \}$ 
denote collections of sources and targets
in the unit cell $\C$ and let $\P_2^{south}$ denote the 
$N_T \times N_S$ block matrix which is the periodizing
operator for all ``south" sources for the pressure for both
the modified Stokes and Stokes equations.
Given a precision $\epsilon$, let $M$ be given by
\cref{mpk0s_lap}.
With $\alpha_m, Q_m$ given in \cref{alphaqdef},
let 
$\bL^{south}, \bL^{north} \in \mathbb{C}^{N_T \times (2M+1)}$ and 
$\bR^{south}, \bR^{north} \in \mathbb{C}^{(2M+1) \times 2N_S}$ be dense matrices
and let 
\[ \bD^{south}, \bD^{north} \in \mathbb{C}^{(2M+1) \times (2M+1)} \]
be diagonal matrices with
\be\label{lowranksouth_stokes_pressure}
\begin{split}
&\bL^{south}(l,m) \ \ =
  e^{-|\alpha_m| y_l} e^{i \alpha_m x_l}, \quad {\rm for}\ m \neq 0, \\
&\bL^{north}(l,m) \ \ =
e^{|\alpha_m| y_l} e^{i \alpha_m x_l}, \quad {\rm for}\ m \neq 0, \\
&\bL^{south}(l,0) \ \ =   
y_l, \quad
\bL^{north}(l,0) \ \ =   
y_l, \\
&\bR^{south}(m,j)\  =  
e^{|\alpha_m| y'_j} \, e^{- i \alpha_m x'_j} [i\alpha_m\,\, -|\alpha_m|]
\quad {\rm for}\ m \neq 0, \\
&\bR^{north}(m,j)\  =  
e^{-|\alpha_m| y'_j} \, e^{- i \alpha_m x'_j} [i\alpha_m\,\, |\alpha_m|]
\quad {\rm for}\ m \neq 0, \\
&\bR^{south}(0,j) \ \ =   
[0\,\, 1], \quad
\bR^{north}(0,j) \ \ =   
[0\,\, 1], \\ 
&\bD^{south}(m,m) = 
-\frac{1}{4\pi |m|} 
\frac{e^{-2Q_{m}}}{1 - e^{-Q_{m}}}  \quad {\rm for}\ m \neq 0 \\
&\bD^{south}(0,0) \ \  =  
\frac{1}{2d\eta}, \quad \bD^{north} = \overline{\bD^{south}}.
\end{split}
\ee
Then
\be
\ba
\P_2^{south} &= \bL^{south} \, \bD^{south} \, \bR^{south} + O(\epsilon),\\
\P_2^{north} &= \bL^{north} \, \bD^{north} \, \bR^{north} + O(\epsilon).
\ea
\ee
\end{theorem}
\begin{theorem} \label{west_thm_stokes_pressure}
Under the hypotheses of \cref{west_thm_lap},
let 
$\bL_1^{west}, \bL_1^{east} \in \mathbb{C}^{N_T \times N^1_q}$,
$\bL_2^{west}, \bL_2^{east} \in \mathbb{C}^{N_T \times N^2_q}$,
be dense matrices and 
let $\bD_1^{e/w}$, $\bD_2^{west}$ and $\bD_2^{east}$ 
be diagonal matrices of dimension $N^1_q$ and $N^2_q$
defined in \cref{lowrankwest_lap}. Let
$\bR_1^{west}, \bR_1^{east} \in \mathbb{C}^{N^1_q \times 2N_S}$,
$\bR_2^{west}, \bR_2^{east} \in \mathbb{C}^{N^2_q \times 2N_S}$
be dense matrices with
\be\label{lowrankwest_stokes_pressure}
\ba
\bR_1^{west}(n,j) &=  
e^{\lambda_{n,1} x'_j} \, e^{-i \lambda_{n,1} y'_j}\, [\lambda_{n,1}\quad -i\lambda_{n,1}] \, , \\
\bR_2^{west}(n,j) &=  
e^{\lambda_{n,2} x'_j} \, e^{-i \lambda_{n,2} y'_j} \,[\lambda_{n,2}\quad -i\lambda_{n,2}] \, ,  \\
\bR_1^{east}(n,j) &=  
e^{-\lambda_{n,1} x'_j} \, e^{-i \lambda_{n,1} y'_j} \,[-\lambda_{n,1}\quad -i\lambda_{n,1}] \, , \\
\bR_2^{east}(n,j) &=  
e^{-\lambda_{n,2} x'_j} \, e^{-i \lambda_{n,2} y'_j} \,[-\lambda_{n,2}\quad -i\lambda_{n,2}] \, .
\ea
\ee
Let
\[
\ba
\P_1^{west} &= \bL_1^{west} \, \bD_1^{e/w} \, \bR_1^{west}, \\
\P_2^{west} &= \bL_2^{west} \, \bD_2^{west} \, \bR_2^{west}. \\
\P_1^{east} &= \bL_1^{east} \, \bD_1^{e/w} \, \bR_1^{east}, \\
\P_2^{east} &= \bL_2^{east} \, \bD_2^{east} \, \bR_2^{east}. \\
\ea
\]
Then the real parts of the vectors 
\[  \P_1^{west} \bq, \ \P_1^{east} \bq, \ 
  P_2^{west} \bq, \P_2^{east} \bq 
\]
denote the contributions from the west or east sources to the 
corresponding periodizing pressures.
\end{theorem}
\Bk

\section{Direct and NUFFT-accelerated methods
for periodizing operators} \label{sec:fastalg}

The low-rank factorizations in the preceding sections
provide a simple fast algorithm for imposing periodic boundary
conditions. It is easy to see that applying the operators from 
right to left in expression of the form
\[ \P \bq = \bL \bD \bR  \bq  \]
requires
$O(r(N_S+N_T))$ work, where $r$ is
the rank of $\P$ (and the dimension of $\bD$).
Because the rank $r$ grows
linearly with the aspect ratio $A=d/\eta$, we describe a more 
involved method
which uses the NUFFT to achieve a computational complexity
of the order $O(\log(1/\epsilon)(r\log r+(N_S+N_T)\log(1/\epsilon)))$.

\begin{remark}
In the singly periodic case, a fast algorithm is required when
the height of the unit cell is much greater than its width - that is,
when $A=d/\eta \ll 1$. Recall that in the doubly periodic case, 
we have defined the orientation of the unit cell so that $A>1$ and
a fast algorithm is needed only when $A \gg 1$.
\end{remark}

\subsection{NUFFT acceleration}

To be concrete, we focus here on the matrix-vector products
\[
\bc = \bR^{south} \, \bq,\,
\bw = \bD^{south} \, \bc,\,
\bu = \bL^{south} \, \bw
\]
for the modified Hemholtz equation 
in \cref{south_thm}, so that $\bu = \P^{south} \bq$
for a unit cell with large aspect ratio.
Before turning to a general distribution of sources, 
let us consider
the case where all sources have the same
$y$-coordinate: 
$\bS = \{ \bs_j \, | \, j = 1,\dots,N_S \}$ with
$\bs_j = (x'_j,\overline{y})$.

Focusing again on the ``south" sources, we have
$\bc =  \bR^{south} \bq$ with
\[ 
c_m = e^{-\chi_m \overline{y}} 
\sum_{j=1}^{N_S} e^{i \alpha_m x_j} q_j.
\]
This is a sum of precisely the form 
\cref{eq:nufft1} and can be computed in $O(r\log r+N_S \log(1/\epsilon))$
work using the NUFFT, where $r=2M+1$ is the rank of $\bR^{south}$.

The next thing to notice is that the entries of $\bR^{south}$ in the
general case involves non-oscillatory functions in the $y$-direction.
In fact, if we define the function $f(y) = e^{-\chi y}$, and
assume $f(y)$ is given at Gauss-Legendre nodes 
$\{ \overline{y}_1, \dots, \overline{y}_{M_{\rm GL}} \}$, then
\[ 
e^{-\chi_m y'_j} \approx
\sum_{n=1}^{M_{\rm GL}} \gamma(y'_j,n) 
e^{-\chi_m \overline{y}_n} \; ,
\]
with spectral accuracy, where
\[ 
\gamma(t,n) = \frac{\frac{\sigma_n}{t-\overline{y}_n}}
      {\sum_{\ell=1}^{M_{\rm GL}} \frac{\sigma_\ell}{t-\overline{y}_\ell}} 
\]
are the interpolation coefficients and
the weights $\sigma_\ell$ are defined as in
\eqref{eq:baryweights}.
Thus, we may write
\be \label{fastanterp} 
\ba
c_m &= 
\sum_{j=1}^{N_S} 
e^{-\chi_m y'_j} e^{i \alpha_m x'_j} q_j \\
&\approx \sum_{j=1}^{N_S} 
\sum_{n=1}^{M_{\rm GL}} {\gamma}(y'_j,n) 
e^{-\chi_m \overline{y}_n} e^{i \alpha_m x'_j} q_j \\
&= 
\sum_{n=1}^{M_{\rm GL}} 
e^{-\chi_m \overline{y}_n} \sum_{j=1}^{N_S} e^{i \alpha_m x'_j} 
[ {\gamma}(y'_j,n) q_j].
\ea
\ee

Thus, by carrying out a total of $M_{\rm GL}$ applications of the NUFFT,
we can obtain $\bc$ with $O(M_{\rm GL} (2M+1))$ additional work (the outer
loop in the last equation of \cref{fastanterp} carried out for each $m$).
 
The treatment of $\bu = \bL^{south} \bw$ is nearly the same.
Using the interpolation formula
\[ 
e^{\chi_m y_l} \approx
\sum_{n=1}^{M_{\rm GL}} {\gamma}(y_l,n) 
e^{\chi_m \overline{y}_n} \, ,
\]
we have
\be \label{fastinterp} 
\ba
u_l &= 
\sum_{m=-M}^{M} 
e^{\chi_m y_l} e^{i \alpha_m x_l} w_m \\
&\approx 
\sum_{m=-M}^{M} 
\sum_{n=1}^{M_{\rm GL}} {\gamma}(y_l,n) e^{\chi_m \overline{y}_n} 
e^{i \alpha_m x_l} \\
&= 
\sum_{n=1}^{M_{\rm GL}} {\gamma}(y_l,n)  
\sum_{m=-M}^{M} 
e^{i \alpha_m x_j}  [w_m e^{-\chi_m \overline{y}_n}].
\ea
\ee

Again,
by carrying out a total of $M_{\rm GL}$ applications of the NUFFT,
we obtain $\bu$ with $O(M_{\rm GL} N_T)$ additional work (the outer
loop in the last equation in \cref{fastinterp}, 
carried out for each $l = 1,\dots,N_T$).
The reader will note that the fast application of $\bL$ is essentially
that of computing the potential on a sequence of horizontal lines in the 
unit cell, followed by interpolation in the $y$-direction.
Because it is the adjoint of the interpolation matrix that is used
in applying $\bR$, that dual process is sometimes called {\em anterpolation}.

The application of $\bL$ and $\bR$ for all of the operators described
in the preceding section is essentially the same, and illustrated in
\cref{fig:nufftgrids}. 

It remains only to estimate the number of interpolation nodes needed,
addressed in the following theorem.

\begin{theorem} \label{nufft_interpolation_thm}
  Suppose that the Green's function $G(\bt,\bs)$ is real analytic
  for $\bt\ne \bs$. Then,
  as a function of $y$ (that is, the $y$-coordinate of the target point $\bt$),
  the kernels $K_2^{south}(\bt,\bs)$,
  $K_2^{north}(\bt,\bs)$ can be well approximated by their
  interpolating polynomials $p_{\rm GL}[K_2^{south}]$, $p_{\rm GL}[K_2^{north}]$
  using Gauss-Legendre interpolation nodes and the following error estimates hold:
  \be
  \ba
  \|K_2^{south}(\bt,\bs)-p_{\rm GL}[K_2^{south}](\bt,\bs)\|&\le C \rho_0^{-M_{\rm GL}},\\
  \|K_2^{north}(\bt,\bs)-p_{\rm GL}[K_2^{north}](\bt,\bs)\|&\le C \rho_0^{-M_{\rm GL}},
  \ea\label{eq:interpolationerror}
  \ee
  for $y\in [-\eta/2,\eta/2]$. Here
  \be
  \rho_0 = 3+\sqrt{8}\approx 5.828.
  \ee
  The same estimates hold for the interpolation errors when both kernels
  are approximated by interpolating polynomials using Gauss-Legendre
  interpolation nodes for the $y$-coordinate of the source $\bs$, 
  which we denote by $y'$, for $y'\in [-\eta/2,\eta/2]$.
\end{theorem}
\begin{proof}
  We will only prove the target interpolation result for $K_2^{south}(\bt,\bs)$,
  since the proofs of the other three cases are almost identical.
  By the definition of $K_2^{south}(\bt,\bs)$ in \cref{ksplitting2}, all
  image sources are separated from the fundamental unit cell by at least
  one cell. That is, for any target $\bt$ in the fundamental unit cell
  with $y\in [-\eta/2,\eta/2]$, the closest image source in the infinite double 
  sum is at $-3\eta/2$. Rescaling the interval $[-\eta/2,\eta/2]$ to
  the standard interval $[-1,1]$, we observe that as a 
  function of $\tilde{y}=2y/\eta$,
  the closest singularity of $K_2^{south}(\bt,\bs)$
  is at $-3$. That is, the Bernstein ellipse with foci at $\pm 1$ in this case has
  semi-major axis length is $a=3$, from which we determine the
  semi-minor axis length to be $b=\sqrt{a^2-c^2}=\sqrt{3^2-1^2}=\sqrt{8}$. 
  The result follows from \cref{eq:interpbound}.
\end{proof}

\begin{remark}
  As discussed in \cref{sec:interpolation}, the constant 
  $C$ in \cref{eq:interpolationerror}
  is equal to $\|G\|_\infty$ in the closed domain bounded by the Bernstein ellipse.
  Since most Green's functions are singular when $\bt=\bs$, 
  $\|G\|_\infty$ is unbounded on the Bernstein ellipse. To make the error
  bound useful, it suffices 
  to shrink the Bernstein ellipse a little to make $C$ finite. 
  In practice, the convergence rate is typically very close
  to what is stated in \cref{nufft_interpolation_thm} and interpolation using
  $8$ or $16$ Legendre nodes leads to six or twelve digit accuracy, respectively.
\end{remark} \Bk    

\begin{figure}
  \centering
  \includegraphics[width=0.8\textwidth]{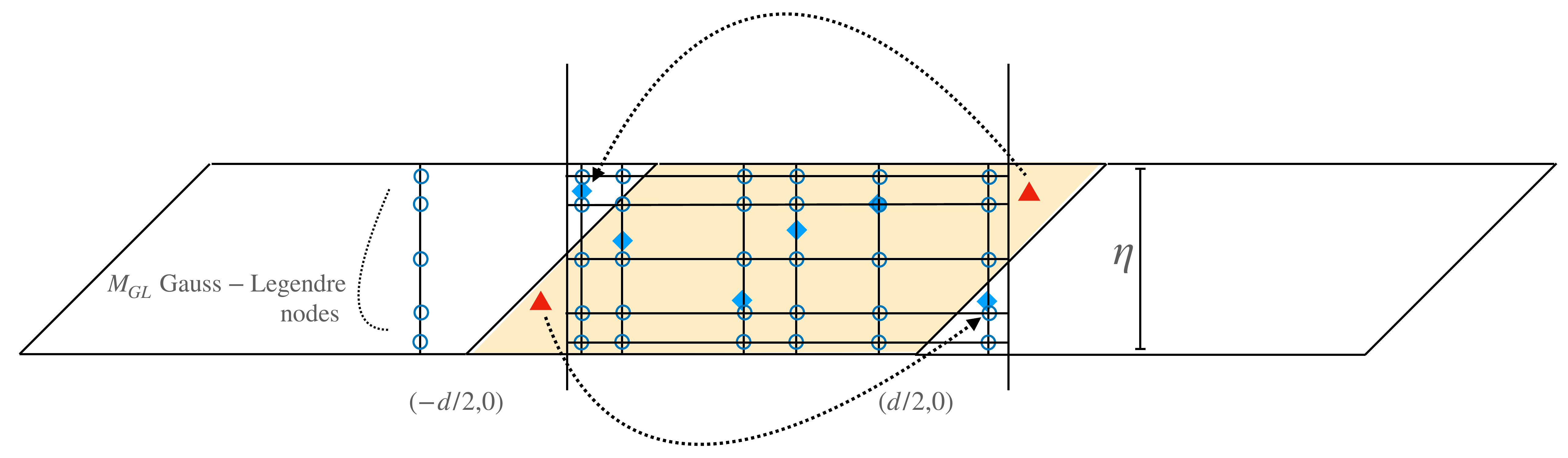}
  \caption{\label{fig:nufftgrids} An illustration of the
    auxiliary grids used for the accelerated algorithm.
    Blue diamonds represent source locations.
    We wrap points (red triangles) that fall outside the 
    rectangular box of dimenson $d \times \eta$ (centered at the unit
    cell center) to their corresponding images within the rectangle.
    The blue circles are the $M_{\rm GL}$ scaled Gauss-Legendre nodes 
    on $[-\eta/2,\eta/2]$ with the same $x$-coordinates as the 
    original sources themselves. In the first step of the 
    method, we use the adjoint of the one-dimensional 
    interpolation matrix to create fictitious sources 
    whose $x$-coordinate is that of the original source
    but whose $y$-coordinate is one of the $M_{\rm GL}$
    Gauss-Legendre nodes at a cost of 
    $O(N_S M_{\rm GL})$ work. The NUFFT then provides a fast
    algorithm for computing $\bR^{south} \bq$
    at a cost of $O(M_{\rm GL} \cdot [(r+N_S)\log(r+N_S)])$ work.
    Likewise, once 
    $\bc = \bD^{south} \, \bR^{south} \bq$ is obtained, the NUFFT
    can be used to evaluate the potential on a tensor-product
    grid with $x$-coordinates corresponding to target locations
    and $y$-coordinates given by the $M_{\rm GL}$ Gauss-Legendre nodes. 
    (wrapped to the rectangular cell) at a cost of 
    $O(M_{\rm GL} \cdot [(r+N_T)\log(r+N_T)])$ work.
    Interpolation yields the field at the desired target points 
    at a cost of $O(N_T M_{\rm GL})$ work.
  }
\end{figure}
\section{Numerical results}  \label{sec:results}
We have implemented the algorithms described in this paper in Fortran. 
Our implementation
uses the fmm2d library \cite{fmm2dlib} 
for the free-space FMMs and the finufft package 
\cite{finufftlib,finufft}
for the NUFFTs. The code
is complied using gfortran 9.3.0 with -O3 option. The results shown in this section
were obtained on a single core of a laptop with Intel(R) 2.40GH i9-10885H CPU.

We first test the performance of the code in the high accuracy regime. \Cref{table1}
shows the results for the modified Helmholtz kernel with precision set
to $10^{-12}$. $40,000$ source points are placed 
in the fundamental unit cell with a uniform random distribution, with
$500$ equispaced target points on each side of the unit cell to check the
enforcement of periodic conditions. 
In the table, $A$ is the aspect ratio (\cref{aspectdef}),
$t_{\rm per}$ is the time for applying the periodizing operator, 
$t_{\rm FMM}$ is the time
for the FMM call with sources in the near region $\N$ consisting of
$(2m_0+1)\times (2n_0+1)$ copies of the unit cell.
$n_0=1$ in the doubly periodic case and 
$n_0=0$ in the singly periodic case. $m_0=1,2,$ or $3$, depending on the
precise shape of the unit cell.
$t_{\rm total}$ is the total computational time and
$t_{\rm FMM}^0$ is the time required by the free-space FMM, with sources restricted
to the fundamental unit cell alone for reference as a lower bound.
All times are measured in seconds and the error
is the estimated relative $l^2$ error in satisfying periodicity
(i.e., the potential difference between the right and left sides for 
the singly periodic case, and the sum of potentials differences in both $x$
and $y$ for the doubly periodic case). $P_1$ and $P_2$ denote the imposition
of periodicty in one or two dimensions, respectively.
For the singly periodic case, $m_0=1$. 
That is, the central $3$ cells are include in the near region.
For the doubly
periodic case, $m_0=1$ for the rectangular cell; $m_0=2$ for the
parallelogram with $\theta=\pi/3$;
and $m_0=3$ for the
parallelogram with $\theta=\pi/6$, where $\theta$ is the angle between 
$\eh_1$ and $\eh_2$.
The cost of the periodization step is insensitive to the geometry
of the unit cell, since we make use of acceleration with the NUFFT, and a 
small fraction of the total cost.
The FMM for sources in the near region $\N$ is about one to four times 
more expensive than for the unit cell alone. 
In our current implemnetation,
we simply call the free-space FMM with all near region 
sources but with targets
restricted to the unit cell. A more efficient code could be developed by 
taking advantage of the fact that the sources in each image cells are identical,
as are the corresponding hierarchy of multipole moments.
Minor modification of the FMM could reduce the cost to being within a factor of
two of the FMM cost for the unit cell alone. 

\begin{table}[t]
\caption{Timing results of the periodic FMM for the modified Helmholtz kernel 
with $\beta=1$ and $40,000$ sources in the unit cell.
The requested precision is $\epsilon=10^{-12}$.
}
\sisetup{
  tight-spacing=true
}
\centering
\begin{tabular}{l
    S[scientific-notation = fixed,fixed-exponent = 0]
    S[scientific-notation = fixed,fixed-exponent = 0]
    S[scientific-notation = fixed,fixed-exponent = 0]
    S[scientific-notation = fixed,fixed-exponent = 0]
    S[scientific-notation = true,table-format=4.2e-2]}
\toprule
${A}$  & ${t_{\rm per}}$ & ${t_{\rm FMM}}$ & ${t_{\rm total}}$ & ${t_{\rm FMM}^0}$ & {Error}  \\
\midrule
$P_1:$\, \text{rectangle}\\
\midrule
   1    &    0.08  &    1.78    &    1.87  &    1.78    &  0.12e-12 \\
  10    &    0.08  &    1.95    &    2.03  &    1.61    &  0.69e-14 \\
 100    &    0.12  &    2.19    &    2.30  &    1.37    &  0.86e-15 \\
1000    &    0.48  &    2.88    &    3.36  &    1.35    &  0.16e-14 \\
\midrule
$P_2:$\, \text{rectangle}\\
\midrule
   1    &    0.18  &    2.45    &    2.63  &    1.69    &  0.31e-13 \\
  10    &    0.17  &    2.48    &    2.65  &    1.62    &  0.44e-14 \\
 100    &    0.17  &    3.70    &    3.87  &    1.39    &  0.49e-15 \\
1000    &    0.19  &    3.72    &    3.91  &    1.36    &  0.48e-15 \\
\midrule
$P_2:$\, \text{parallelogram with}\, $\theta=\pi/3$\\
\midrule
   2    &    0.18  &    3.40    &    3.59  &    1.82    &  0.20e-12 \\
  10    &    0.17  &    3.45    &    3.63  &    1.78    &  0.35e-12 \\
 100    &    0.17  &    4.24    &    4.42  &    1.39    &  0.12e-12 \\
1000    &    0.19  &    3.84    &    4.03  &    1.36    &  0.45e-12 \\
\midrule
$P_2:$\, \text{parallelogram with}\, $\theta=\pi/6$\\
\midrule
   2    &    0.18  &    4.18    &    4.37  &    1.53    &  0.17e-12 \\
  10    &    0.17  &    3.30    &    3.48  &    1.95    &  0.15e-12 \\
 100    &    0.17  &    4.25    &    4.42  &    1.38    &  0.32e-12 \\
1000    &    0.22  &    3.95    &    4.18  &    1.36    &  0.26e-12 \\
\bottomrule
\end{tabular}
\label{table1}
\end{table}

Similar results hold for the other kernels. In
\cref{table2} we show the timings obtained for the Laplace kernel with
precision $\epsilon=10^{-9}$, and in \cref{table3}, we show the timings 
obtained for the Stokeslet with precision $\epsilon=10^{-6}$. 
The column headings have the same meaning as in \cref{table1}.

\begin{table}[t]
\caption{Timing results of the periodic FMM for the Laplace kernel with
$40,000$ sources in the unit cell and a
requested precision of $\epsilon=10^{-9}$.
}
\sisetup{
  tight-spacing=true
}
\centering
\begin{tabular}{l
    S[scientific-notation = fixed,fixed-exponent = 0]
    S[scientific-notation = fixed,fixed-exponent = 0]
    S[scientific-notation = fixed,fixed-exponent = 0]
    S[scientific-notation = fixed,fixed-exponent = 0]
    S[scientific-notation = true,table-format=4.2e-2]}
\toprule
${A}$  & ${t_{\rm per}}$ & ${t_{\rm FMM}}$ & ${t_{\rm total}}$ & ${t_{\rm FMM}^0}$ & {Error}  \\
\midrule
$P_1:$\, \text{rectangle}\\
\midrule
   1    &    0.06  &    0.56    &    0.62  &    0.64    &  0.10D-09 \\
  10    &    0.07  &    0.83    &    0.90  &    0.66    &  0.98D-11 \\
 100    &    0.08  &    0.64    &    0.73  &    0.44    &  0.21D-12 \\
1000    &    0.23  &    1.02    &    1.25  &    0.40    &  0.15D-12 \\
\midrule
$P_2:$\, \text{rectangle}\\
\midrule
   1    &    0.14  &    1.03    &    1.17  &    0.54    &  0.11D-10 \\
  10    &    0.12  &    0.81    &    0.93  &    0.66    &  0.63D-11 \\
 100    &    0.12  &    1.14    &    1.27  &    0.44    &  0.13D-11 \\
1000    &    0.13  &    1.48    &    1.61  &    0.40    &  0.52D-13 \\
\midrule
$P_2:$\, \text{parallelogram with}\, $\theta=\pi/3$\\
\midrule
   2    &    0.14  &    1.13    &    1.27  &    0.65    &  0.26D-10 \\
  10    &    0.12  &    1.17    &    1.29  &    0.72    &  0.13D-09 \\
 100    &    0.12  &    1.64    &    1.77  &    0.44    &  0.31D-09 \\
1000    &    0.13  &    1.29    &    1.42  &    0.40    &  0.24D-09 \\
\midrule
$P_2:$\, \text{parallelogram with}\, $\theta=\pi/6$\\
\midrule
   2    &    0.13  &    1.76    &    1.90  &    0.46    &  0.20D-10 \\
  10    &    0.13  &    1.11    &    1.24  &    0.62    &  0.23D-09 \\
 100    &    0.12  &    1.66    &    1.78  &    0.45    &  0.60D-10 \\
1000    &    0.13  &    1.28    &    1.42  &    0.40    &  0.34D-09 \\
\bottomrule
\end{tabular}
\label{table2}
\end{table}

\begin{table}[t]
\caption{Timing results of the periodic FMM for the Stokeslet with
$40,000$ sources in the unit cell and a
requested precision of $\epsilon=10^{-6}$.
}
\sisetup{
  tight-spacing=true
}
\centering
\begin{tabular}{l
    S[scientific-notation = fixed,fixed-exponent = 0]
    S[scientific-notation = fixed,fixed-exponent = 0]
    S[scientific-notation = fixed,fixed-exponent = 0]
    S[scientific-notation = fixed,fixed-exponent = 0]
    S[scientific-notation = true,table-format=4.2e-2]}
\toprule
${A}$  & ${t_{\rm per}}$ & ${t_{\rm FMM}}$ & ${t_{\rm total}}$ & ${t_{\rm FMM}^0}$ & {Error}  \\
\midrule
$P_1:$\, \text{rectangle}\\
\midrule
   1    &    0.09  &    1.05    &    1.14  &    0.81    &  0.19D-06 \\
  10    &    0.07  &    1.17    &    1.24  &    0.94    &  0.12D-06 \\
 100    &    0.08  &    1.09    &    1.18  &    0.78    &  0.16D-08 \\
1000    &    0.19  &    1.31    &    1.50  &    0.69    &  0.39D-10 \\
\midrule
$P_2:$\, \text{rectangle}\\
\midrule
   1    &    0.25  &    1.63    &    1.89  &    0.69    &  0.21D-07 \\
  10    &    0.24  &    1.31    &    1.55  &    0.94    &  0.53D-07 \\
 100    &    0.23  &    1.42    &    1.65  &    0.79    &  0.73D-08 \\
1000    &    0.24  &    2.30    &    2.55  &    0.69    &  0.12D-09 \\
\midrule
$P_2:$\, \text{parallelogram with}\, $\theta=\pi/3$\\
\midrule
   2    &    0.25  &    1.50    &    1.75  &    0.78    &  0.35D-07 \\
  10    &    0.24  &    1.55    &    1.80  &    0.88    &  0.19D-06 \\
 100    &    0.24  &    2.30    &    2.55  &    0.80    &  0.57D-06 \\
1000    &    0.25  &    2.31    &    2.56  &    0.69    &  0.39D-06 \\
\midrule
$P_2:$\, \text{parallelogram with}\, $\theta=\pi/6$\\
\midrule
   2    &    0.25  &    2.24    &    2.50  &    0.88    &  0.72D-07 \\
  10    &    0.24  &    1.58    &    1.83  &    0.77    &  0.47D-06 \\
 100    &    0.24  &    2.29    &    2.53  &    0.82    &  0.32D-06 \\
1000    &    0.25  &    2.27    &    2.52  &    0.68    &  0.61D-06 \\
\bottomrule
\end{tabular}
\label{table3}
\end{table}

\section{Conclusions}  \label{sec:conclusions}

Explicit, separable low-rank factorizations have been constructed for the
periodizing operator for particle interactions governed by the modified 
Helmholtz, Poisson,
modified Stokes, and Stokes equations in two dimensions. The factorization is
based on the Sommerfeld integral representation of the Green's function, 
which is readily available
for the modified Helmholtz and Poisson kernels, and can be derived more
generally by Fourier analysis and contour integration, 
as done here for the modified Stokeslet or Stokeslet. In both the singly
and doubly periodic cases, the $\epsilon$-rank $r$ of the periodizing 
operator is shown to be of the order
$O\left(\log(1/\epsilon)\left(\log(1/\beta)+A\log(1/\epsilon)\right)\right)$,
where $A$ is the aspect ratio of the fundamental unit cell. Here, 
$\beta$ is the parameter that defines
the modified Helmholtz and modified Stokes kernels.
For the Poisson and Stokes kernels, the 
factor $\log(1/\beta)$ disappears.

Our factorization leads to a simple fast algorithm for the action of 
the periodizing
operators with $O(r(N_T+N_S))$ complexity - linear with respect to the number
of targets and sources. When $r$ is large,  a more complicated
fast algorithm, relying on the NUFFT, can be used to further speed up
the calculation, reducing the complexity to $O(\log(1/\epsilon)(r\log r+(N_T+N_S)\log(1/\epsilon)))$.

There are several natural extensions or generalizations of the current work.
First, the scheme can easily be extended to treat nonoscillatory kernels in 
three dimensions. Second, 
there is no essential obstacle to extending the scheme to treat
oscillatory problems (such as the Helmholtz or Maxwell equations) 
in two and three dimensions.
The various sums and integrals, however, must be treated with more care,
as they are conditionally convergent, permit ``quasi-periodic" 
boundary conditions and are subject to resonances (Wood anomalies)
\cite{barnett2010jcp,barnett2011bit,denlinger2017jmp,dienstfrey2001prsl,enoch2001,mcphedran2000jmp}.
Third, the scheme can be coupled with integral equation methods
and the fast multipole method to solve periodic boundary value problems
when the unit cell contains inclusions of complicated shape.
Finally, more efficient versions of the FMM can be deployed to 
reduce the cost of handling the near region copies of the unit cell, 
as in the periodic version of the original scheme~\cite{greengard1987jcp}.
This would bring into closer alignment the time 
${t_{\rm FMM}}$ and  ${t_{\rm FMM}^0}$ in \cref{table1,table2,table3}.
For multiple scattering
problems with singly or doubly periodic boundary conditions, where
the far field of a scatterer is represented by a multipole expansion,
the periodic scattering matrix can be constructed via simple 
modifications of the algorithms in \cite{gan2016sisc,fmps2013jcp}.
This requires periodizing operators for multipole sources, which are presented
in the appendices of the present paper.
\section*{Acknowledgments}
The authors would like to thank Jingfang Huang at the University of
North Carolina at Chapel Hill, Alex Barnett and Manas Rachh at the
Flatiron Institute for helpful discussions.

\appendix
\section{Rotated plane-wave expansions for the east and west parts
  of the doubly periodic periodizing operators}

In the analysis and implementation of the present paper,
we have relied on plane-wave expansions 
that decay in $x$: either for $x>0$ 
(the west part) or for $x<0$ (the east part). 
Simple geometric considerations led to the conclusion that
we may need to exclude the central $7\times 3$ copies of the unit 
cell. For non-rectangualr unit cells, it is actually more efficient to
align the decay direction in the plane-wave expansion with
$\eh_2^\perp$ - that is, orthogonal to the $\eh_2$ direction.
We illustrate the corresponding algorithm in the case of the modified
Helmholtz kernel. Consider the
coordinate transformation
\be
\begin{pmatrix} \tx\\ \ty\end{pmatrix} =
\begin{pmatrix} \cos \theta& -\sin \theta \\ \sin \theta & \cos \theta 
\end{pmatrix} 
\begin{pmatrix} x\\y\end{pmatrix}.
\ee
In complex notation, this is equivalent to
\be
\tx+i \ty=e^{i\theta} (x+iy).
\ee
Let us also write
\be
\txi+\teta=e^{i\theta}(\xi+i\eta), \quad \td_x+i\td_y=d e^{i\theta}.
\ee
For the west part, if the plane-wave expansion along the $\tx$ direction 
is used, we have
\be\label{eq:k0west2}
\ba
K_2^{west}(\bt,\bs)&= \frac{1}{2\pi} \sum_{m=-\infty}^{-(m_0+1)}\sum_{n=-1}^1
K_0\left(\bt,\bs+\bl_{mn}\right)\\
&=\sum_{m=-\infty}^{-(m_0+1)}\sum_{n=-1}^1
 \int_{-\infty}^\infty
\frac{e^{-\sqrt{\lambda^2+\beta^2}(\tx-\tx'-m\td_x-n\txi)}}{4 \pi \sqrt{\lambda^2+\beta^2}}
\cdot e^{i\lambda (\ty-\ty'-m \td_y-n\teta)} d\lambda\\
&= \sum_{n=-1}^1\int_{-\infty}^\infty
\frac{e^{-\sqrt{\lambda^2+\beta^2}(\tx-\tx'-n\txi)}}{4 \pi \sqrt{\lambda^2+\beta^2}}
\cdot e^{i\lambda (\ty-\ty'-n\teta)} 
\frac{e^{-(m_0+1)\left(\sqrt{\lambda^2+\beta^2}\td_x-i\lambda \td_y\right)}}
     {1-e^{-\left(\sqrt{\lambda^2+\beta^2}\td_x-i\lambda \td_y\right)}}
     d\lambda.
\ea
\ee
\begin{figure}
  \centering
  \includegraphics[width=0.8\textwidth]{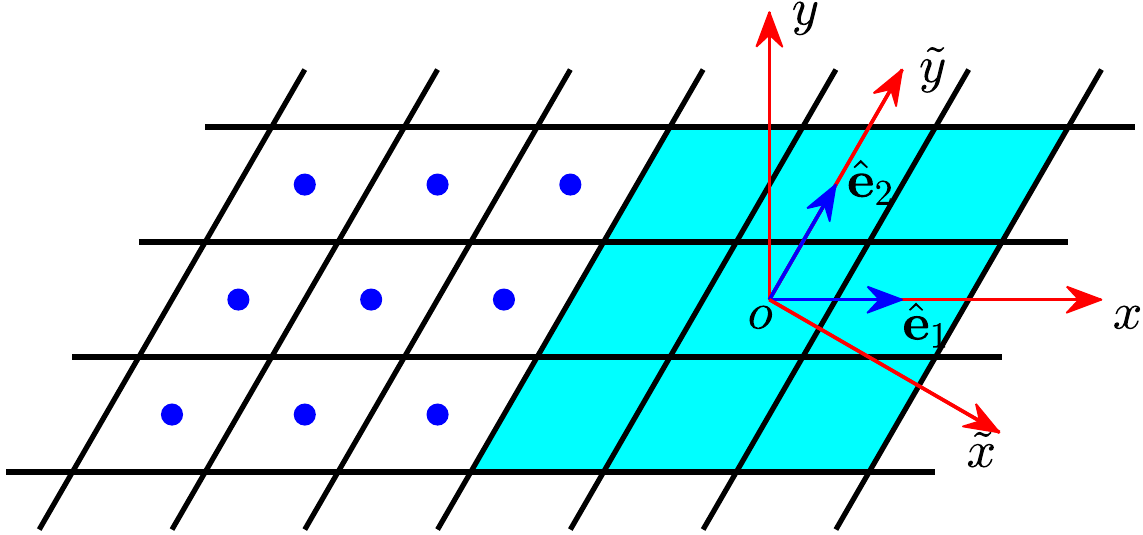}
  \caption{\label{fig:newplanewavedirection}
    New direction of the plane-wave expansion for the west part. In the main text,
    we have chosen the plane-wave expansions along the coordinate axes for all four
    parts. The advantage is that the east and west parts of the doubly periodic periodizing
    operators can be discretized via efficient precomputed generalized Gaussian quadrature. But the
    worst case requires the exclusion of the center $7\times 3$ cells from the periodizing
    operators. If we choose the plane-wave expansion along the $\tilde{x}$-axis, then
    one only needs to exclude the center $3\times 3$ cells from the doubly periodic
    periodizing operators. But the number of plane waves may increase if the angle between
    $\eh_1$ and $\eh_2$ is very small and $|\eh_2|$ is very close to $|\eh_1|$.
  }
\end{figure}

It is now clear that if we choose ${\bf e}_{\tx}=\eh_2^\perp$ - that is,
we choose
$\theta$ such that $\txi=0$ and $\td_x>0$, then $m_0=1$ is sufficient to ensure
that the decaying exponential in the integrand decays at least as fast
as $e^{-\sqrt{\lambda^2+\beta^2}\td_x}$. Thus, 
one only needs to exclude the center $3\times 3$ cells from the periodizing 
operator rather than the larger near region we have used above. 
The integrand could still
be highly oscillatory, so that an effective high-order quadrature
is needed, just as in singly periodic case.
\section{Periodizing operators for the modified Helmholtz equation
with multipole sources}
\label{sec:yukawamultipole}
The multipole of order $l$ for
the modified Helmholtz multipole is defined by $K_l(\beta r)e^{il\theta}$,
where $K_l$ the modified Bessel function of the second kind of order $l$.
The following lemma describes the corresponding plane-wave expansions for the
far-field contributions of the periodizing operators.

\begin{lemma}
For the standard unit cell $\C$ discussed in the main text,
let $K_2^{south}$, $K_2^{north}$, $K_2^{west}$, $K_2^{east}$ denote
the far-field parts of the periodizing operator
for a multipole source of order $l$ governed by the modified Helmholtz equation
subject to doubly periodic boundary conditions.
That is,
\be
\ba
K_2^{south}(\bt,\bs) &=
\sum_{n=-\infty}^{-2}\sum_{m=-\infty}^\infty
K_l(\bt,\bs+\bl_{mn})e^{il\theta_{mn}}, \\
K_2^{north}(\bt,\bs) &=
\sum_{n=2}^\infty \sum_{m=-\infty}^\infty
K_l(\bt,\bs+\bl_{mn})e^{il\theta_{mn}}, \\
K_2^{west}(\bt,\bs) &=
\sum_{n=-1}^{1} \sum_{m=-\infty}^{-4}
K_l(\bt,\bs+\bl_{mn})e^{il\theta_{mn}}, \\
K_2^{east}(\bt,\bs) &=
\sum_{n=-1}^{1} \sum_{m=4}^\infty
K_l(\bt,\bs+\bl_{mn})e^{il\theta_{mn}}.
\ea
\ee
Let $\alpha_m$, $\chi_m$ and $Q_m$ be given by \cref{chialphadef}.
Then
\be
\ba
K_2^{south}(\bt,\bs)
&=\frac{\pi i^l}{d}
\sum_{m=-\infty}^\infty
\left(\frac{\beta}{\chi_m+\alpha_m}\right)^l\frac{1}{\chi_m} e^{-\chi_m(y-y')+i\alpha_m(x-x')}
\frac{e^{-2Q_m}}{1-e^{-Q_m}}\, ,\\
K_2^{north}(\bt,\bs)
&=\frac{\pi (-i)^l}{d}
\sum_{m=-\infty}^\infty
\left(\frac{\chi_m+\alpha_m}{\beta}\right)^l\frac{1}{\chi_m} e^{\chi_m(y-y')+i\alpha_m(x-x')}
\frac{e^{-2\overline{Q_m}}}{1-e^{-\overline{Q_m}}}\, ,\\
K_2^{west}(\bt,\bs)
&=\frac{1}{2 \beta^l} \sum_{n=-1}^1\int_{-\infty}^\infty
\left(\sqrt{\lambda^2+\beta^2}+\lambda\right)^l
\frac{e^{-\sqrt{\lambda^2+\beta^2}(x-x'-n\xi)}}{\sqrt{\lambda^2+\beta^2}}\\
& \cdot e^{i\lambda (y-y'-n\eta)}
\frac{e^{-4\sqrt{\lambda^2+\beta^2}d}}
     {1-e^{-\sqrt{\lambda^2+\beta^2}d}}
     d\lambda\, ,\\
K_2^{east}(\bt,\bs)
&=\frac{(-1)^l}{2 \beta^l} \sum_{n=-1}^1\int_{-\infty}^\infty
\left(\sqrt{\lambda^2+\beta^2}-\lambda\right)^l
\frac{e^{\sqrt{\lambda^2+\beta^2}(x-x'-n\xi)}}{\sqrt{\lambda^2+\beta^2}}\\
& \cdot e^{i\lambda (y-y'-n\eta)}
\frac{e^{-4\sqrt{\lambda^2+\beta^2}d}}
     {1-e^{-\sqrt{\lambda^2+\beta^2}d}}
     d\lambda\, .
\ea
\ee
Similarly, for the singly periodic case,
\be
\ba
K_1^{west}(\bt,\bs)&=\sum_{m=-\infty}^{-2}K_l(\bt,\bs+(md,0))e^{il\theta_{m0}}, \\
&=\frac{1}{2 \beta^l} \int_{-\infty}^\infty
\left(\sqrt{\lambda^2+\beta^2}+\lambda\right)^l
\frac{e^{-\sqrt{\lambda^2+\beta^2}(x-x')}}{\sqrt{\lambda^2+\beta^2}}\\
& \cdot e^{i\lambda (y-y')}
\frac{e^{-2\sqrt{\lambda^2+\beta^2}d}}
     {1-e^{-\sqrt{\lambda^2+\beta^2}d}}
     d\lambda\, ,\\
K_1^{east}(\bt,\bs)&=\sum_{m=2}^\infty
K_l(\bt,\bs+(md,0))e^{il\theta_{m0}}\\
&=\frac{(-1)^l}{2 \beta^l}\int_{-\infty}^\infty
\left(\sqrt{\lambda^2+\beta^2}-\lambda\right)^l
\frac{e^{\sqrt{\lambda^2+\beta^2}(x-x')}}{\sqrt{\lambda^2+\beta^2}}\\
& \cdot e^{i\lambda (y-y')}
\frac{e^{-2\sqrt{\lambda^2+\beta^2}d}}
     {1-e^{-\sqrt{\lambda^2+\beta^2}d}}
     d\lambda\, .
\ea
\ee

\end{lemma}

The preceding result yields the following low-rank decompositions
for the periodizing operators.

\begin{lemma}
Under the hypotheses of \cref{south_thm} and \cref{west_thm}, let
\[ \bL^{south}, \bL^{north} \in \mathbb{C}^{N_T \times (2M+1)} \]
and 
$\bR^{south}, \bR^{north} \in \mathbb{C}^{(2M+1) \times N_S}$ be dense matrices
defined in \cref{lowranksouth3,lowranknorth}, and let
$\bL_1^{west} \in \mathbb{C}^{N_T \times 2N^1_q}$,
$\bL_2^{west} \in \mathbb{C}^{N_T \times 2N^2_q}$,
$\bR_1^{west} \in \mathbb{C}^{2N^1_q \times N_S}$,
$\bR_2^{west} \in \mathbb{C}^{2N^2_q \times N_S}$
be dense matrices defined in \cref{lowrankwest}. Furthermore,
let $\bD^{south}, \bD^{north} \in \mathbb{C}^{(2M+1) \times (2M+1)}$ be diagonal
matrices with
\be\label{lowranksouth_multipole}
\ba
\bD^{south}(m,m) &=
\frac{\pi i^l}{d}
\left(\frac{\beta}{\chi_m+\alpha_m}\right)^l\frac{1}{\chi_m}
\frac{e^{-2Q_m}}{1-e^{-Q_m}}\, ,\\
\bD^{north}(m,m)
&=\frac{\pi (-i)^l}{d}
\left(\frac{\chi_m+\alpha_m}{\beta}\right)^l\frac{1}{\chi_m}
\frac{e^{-2\overline{Q_m}}}{1-e^{-\overline{Q_m}}}\, ,\\
\ea
\ee
and let $\bD_1^{west}, \bD_1^{east}$, and $\bD_2^{west}$, $\bD_2^{east}$
be diagonal matrices of dimension $2N^1_q$ and $2N^2_q$, respectively,
with
\be\label{lowrankwest_multipole}
\ba
\bD_1^{west}(n,n)
&=\frac{1}{2 \beta^l}
\left(\sqrt{\lambda_{n,1}^2+\beta^2}+\lambda_{n,1}\right)^l
\frac{w_{n,1}}{\sqrt{\lambda_{n,1}^2+\beta^2}}
\frac{e^{-2\sqrt{\lambda_{n,1}^2+\beta^2}d}}
     {1-e^{-\sqrt{\lambda_{n,1}^2+\beta^2}d}}\, ,\\
\bD_1^{east}(n,n)&=\frac{(-1)^l}{2 \beta^l}
\left(\sqrt{\lambda_{n,1}^2+\beta^2}-\lambda_{n,1}\right)^l
\frac{w_{n,1}}{\sqrt{\lambda_{n,1}^2+\beta^2}}
\frac{e^{-2\sqrt{\lambda_{n,1}^2+\beta^2}d}}
     {1-e^{-\sqrt{\lambda_{n,1}^2+\beta^2}d}}\, ,\\
\bD_2^{west}(n,n)
&=\frac{1}{2 \beta^l}
\left(\sqrt{\lambda_{n,2}^2+\beta^2}+\lambda_{n,2}\right)^l
\frac{w_{n,2}}{\sqrt{\lambda_{n,2}^2+\beta^2}}
\frac{e^{-4\sqrt{\lambda_{n,2}^2+\beta^2}d}}
     {1-e^{-\sqrt{\lambda_{n,2}^2+\beta^2}d}}\\
&\cdot [ e^{-\sqrt{\lambda_{n,2}^2+\beta^2} \xi + i \lambda_{n,2} \eta} +
  e^{\sqrt{\lambda_{n,2}^2+\beta^2} \xi - i \lambda_{n,2} \eta} + 1]\, ,\\
\bD_2^{east}(n,n)&=\frac{(-1)^l}{2 \beta^l}
\left(\sqrt{\lambda_{n,2}^2+\beta^2}-\lambda_{n,2}\right)^l
\frac{w_{n,2}}{\sqrt{\lambda_{n,2}^2+\beta^2}}
\frac{e^{-4\sqrt{\lambda_{n,2}^2+\beta^2}d}}
     {1-e^{-\sqrt{\lambda_{n,2}^2+\beta^2}d}}\\
&\cdot [ e^{-\sqrt{\lambda_{n,2}^2+\beta^2} \xi - i \lambda_{n,2} \eta} +
  e^{\sqrt{\lambda_{n,2}^2+\beta^2} \xi + i \lambda_{n,2} \eta} + 1]\, ,
\ea
\ee
where $\lambda_{-n,1}=-\lambda_{n,1}$ for $n=1, \ldots,  N_q^1$,
and $\lambda_{-n,2}=-\lambda_{n,2}$ for $n=1, \ldots, N_q^2$.
Then, the periodizing operators for the modified
Helmholtz multipole of order $l$ are given by
\be
\ba
\P_2^{south} &= \bL^{south} \, \bD^{south} \, \bR^{south} + O(\epsilon),\\
\P_2^{north} &= \bL^{north} \, \bD^{north} \, \bR^{north} + O(\epsilon),\\
\P_2^{west} &= \bL_2^{west} \, \bD_2^{west} \, \bR_2^{west} + O(\epsilon),\\
\P_2^{east} &= \bL_2^{east} \, \bD_2^{east} \, \bR_2^{east} + O(\epsilon),\\
\P_1^{west} &= \bL_1^{west} \, \bD_1^{west} \, \bR_1^{west} + O(\epsilon),\\
\P_1^{east} &= \bL_1^{east} \, \bD_1^{east} \, \bR_1^{east} + O(\epsilon).
\ea
\ee
\end{lemma}
\section{Periodizing operators for the Laplace equation with 
multipole sources}
\label{sec:laplacemultipole}
In two dimensions, using complex variables notation, the Laplace 
multipole of order $l$ is simply $1/z^l$.
Here we identify $\bt$ with $z=x+iy$, $\bs$ with $z'=x'+iy'$,
$\eh_1$ with $e_1=d$, $\eh_2$ with $e_2=\xi+i\eta$,
and $\bl_{mn}$ with $z_{mn}=m\cdot e_1 +n \cdot e_2$.
The following lemma contains the plane-wave expansions for the
far-field parts of the corresponding periodic kernels.
\begin{lemma}
For the standard unit cell $\C$ discussed in the main text,
let $K_2^{south}$, $K_2^{north}$, $K_2^{west}$, $K_2^{east}$ denote
the far-field parts of the periodizing operator
for a multipole source of order $l$ governed by the Laplace equation
subject to doubly periodic boundary conditions.
That is,
\be
\ba
K_2^{south}(\bt,\bs) &=
\sum_{n=-\infty}^{-2}\sum_{m=-\infty}^\infty
\frac{1}{(z-z'-z_{mn})^l}, \\
K_2^{north}(\bt,\bs) &=
\sum_{n=2}^\infty \sum_{m=-\infty}^\infty
\frac{1}{(z-z'-z_{mn})^l}, \\
K_2^{west}(\bt,\bs) &=
\sum_{n=-1}^{1} \sum_{m=-\infty}^{-4}
\frac{1}{(z-z'-z_{mn})^l}, \\
K_2^{east}(\bt,\bs) &=
\sum_{n=-1}^{1} \sum_{m=4}^\infty
\frac{1}{(z-z'-z_{mn})^l}.
\ea
\ee
Let $Q_m=2\pi m(\eta-i\xi)/d$.
Then
\be
\ba
K_2^{south}(\bt,\bs)
&=\frac{(-2\pi i)^l}{(l-1)!d^l}\,\sum_{m=1}^{\infty}
m^{l-1} e^{i\frac{2\pi m}{d} (z-z')}\frac{e^{-2Q_m}}{1-e^{-Q_m}} + \delta_{l1}\frac{\pi i}{d\eta}y,\\
K_2^{north}(\bt,\bs)
&=\frac{(2\pi i)^l}{(l-1)!d^l}\,\sum_{m=1}^{\infty}
m^{l-1} e^{-i\frac{2\pi m}{d} (z-z')}\frac{e^{-2\overline{Q_m}}}{1-e^{-\overline{Q_m}}}
+ \delta_{l1}\frac{\pi i}{d\eta}y,\\
K_2^{west}(\bt,\bs)
&=\frac{1}{(l-1)!}\,
\int_0^\infty \lambda^{l-1} \left(1+e^{\lambda\cdot e_2}+e^{-\lambda\cdot e_2}\right)
e^{-\lambda (z-z')}
\frac{e^{-4\lambda d}}{1-e^{-\lambda d}}d\lambda,\\
K_2^{east}(\bt,\bs)
&=\frac{(-1)^l}{(l-1)!}\,
\int_0^\infty \lambda^{l-1} \left(1+e^{\lambda\cdot e_2}+e^{-\lambda\cdot e_2}\right)
e^{\lambda (z-z')}
\frac{e^{-4\lambda d}}{1-e^{-\lambda d}}d\lambda.
\ea\label{lowranksouth_lapmultipole}
\ee
Similarly, for the singly periodic case,
\be
\ba
K_1^{west}(\bt,\bs)&=\sum_{m=-\infty}^{-2}
\frac{1}{(z-z'-z_{mn})^l} \\
&=\frac{1}{(l-1)!}\,
\int_0^\infty \lambda^{l-1} 
e^{-\lambda (z-z')}
\frac{e^{-2\lambda d}}{1-e^{-\lambda d}}d\lambda,\\
K_1^{east}(\bt,\bs)&=\sum_{m=2}^\infty
\frac{1}{(z-z'-z_{mn})^l} \\
&=\frac{(-1)^l}{(l-1)!}\,
\int_0^\infty \lambda^{l-1} 
e^{\lambda (z-z')}
\frac{e^{-2\lambda d}}{1-e^{-\lambda d}}d\lambda.
\ea\label{lowrankwest_lapmultipole}
\ee
\end{lemma}

The derivation of the associated periodizing operators
is straightforward and omitted. Note that the integrals
in \cref{lowranksouth_lapmultipole} and \cref{lowrankwest_lapmultipole}
diverge at the origin when $l=1$, but the divergence is compensated for
in the associated periodizing operators
under the assumption of charge neutrality.

\section{Periodizing operators for the Stokes stresslet}
\label{sec:stresslet}
The stresslet for the Stokes equation is defined by the formula
\be
T^{(S)}_{ijk}(\bt,\bs) = \frac{\partial G^{(S)}_{ij}(\bt,\bs)}{\partial x_k}
+\frac{\partial G^{(S)}_{jk}(\bt,\bs)}{\partial x_i}
-p_j(\bt,\bs)\delta_{jk},
\ee
where $G^{(S)}_{ij}$ is the $ij$-th component of the Stokeslet in \cref{stokeslet},
$p_j$ is the $j$th component of the pressurelet in \cref{pressurelet},
and the partial derivatives are with respect to the source point $\bs$. It is inconvenient
to write down the periodizing operators for the stresslet due to its tensor structure.
In practice, it is often combined with a vector $\bn$ to form the kernel of the double
layer potential operator or its adjoint operator, when $\bn=(n_1,n_2)$ is the unit normal vector
at the source point $\bs$ or the target point $\bt$, respectively. Thus, we will write
down the periodizing operators for the kernel $\bD^{(S)}$ of the double layer
potential operator defined by
the formula $D^{(S)}_{ij}=T^{(S)}_{jik}n_k$ instead.

\begin{lemma}
For the standard unit cell $\C$ discussed in the main text,
let $K_2^{south}$, $K_2^{north}$, $K_2^{west}$, $K_2^{east}$ denote
the far-field parts of the periodizing operator
for the kernel of the Stokes double layer potential
subject to doubly periodic boundary conditions.
That is,
\be
\ba
K_2^{south}(\bt,\bs) &=
\sum_{n=-\infty}^{-2}\sum_{m=-\infty}^\infty \bD^{(S)}(\bt,\bs+\bl_{mn}), \\
K_2^{north}(\bt,\bs) &=
\sum_{n=2}^\infty \sum_{m=-\infty}^\infty \bD^{(S)}(\bt,\bs+\bl_{mn}), \\
K_2^{west}(\bt,\bs) &=
\sum_{n=-1}^{1} \sum_{m=-\infty}^{-4} \bD^{(S)}(\bt,\bs+\bl_{mn}), \\
K_2^{east}(\bt,\bs) &=
\sum_{n=-1}^{1} \sum_{m=4}^\infty \bD^{(S)}(\bt,\bs+\bl_{mn}).
\ea
\ee
Let $\alpha_m=2\pi m/d$ and $Q_m=2\pi m(\eta-i\xi)/d$.
Then
\be
\ba
\K_2^{south}(\bt,\bs)
&=\frac{y}{2d\eta} \begin{bmatrix} n_2 & n_1\\ n_1 & n_2 \end{bmatrix}
+\frac{1}{2d} \sum_{\substack{m=-\infty \\ m\ne 0}}^\infty\left\{
\begin{bmatrix} 2i\sign(m) n_1-n_2 & -n_1 \\ -n_1 & -n_2 \end{bmatrix}\right.\\
&\left.-(i\alpha_m n_1-|\alpha_m| n_2)\left(y-y'+\frac{2-e^{-Q_m}}{1-e^{-Q_m}}\eta\right)
\begin{bmatrix} 1 & i\sign(m) \\ i\sign(m) & -1 \end{bmatrix}\right\}\\
&\qquad\qquad\qquad\cdot \frac{e^{-2Q_m}}{1-e^{-Q_m}}
e^{-|\alpha_m|(y-y')+i\alpha_m(x-x')},\\
\K_2^{north}(\bt,\bs)
&=\frac{y}{2d\eta} \begin{bmatrix} n_2 & n_1\\ n_1 & n_2 \end{bmatrix}
+\frac{1}{2d} \sum_{\substack{m=-\infty \\ m\ne 0}}^\infty\left\{
\begin{bmatrix} 2i\sign(m) n_1+n_2 & n_1 \\ n_1 & n_2 \end{bmatrix}\right.\\
&\left.-(i\alpha_m n_1+|\alpha_m| n_2)
\left(y-y'-\frac{2-e^{-\overline{Q_m}}}{1-e^{-\overline{Q_m}}}\eta\right)
\begin{bmatrix} -1 & i\sign(m) \\ i\sign(m) & 1 \end{bmatrix}\right\}\\
&\qquad\qquad\qquad\cdot \frac{e^{-2\overline{Q_m}}}{1-e^{-\overline{Q_m}}}
e^{|\alpha_m|(y-y')+i\alpha_m(x-x')},
\ea
\ee
\be
\ba
\K_2^{west}(\bt,\bs)&=
\frac{1}{4\pi}\sum_{n=-1}^1 \int_{-\infty}^\infty
\frac{e^{-4|\lambda|d}}{1-e^{-|\lambda|d}} e^{-|\lambda|(x-x'-n\xi)}e^{i\lambda(y-y'-n\eta)}
\cdot\left\{\begin{bmatrix} -n_1 & -n_2 \\ -n_2 & -n_1+2i\sign(\lambda)n_2 \end{bmatrix}\right.\\
&\left.  - (-|\lambda|n_1+i\lambda n_2)
\left(x-x'-n\xi+\frac{4-3e^{-|\lambda|d}}{1-e^{-|\lambda|d}}d\right)
\begin{bmatrix}-1 & i\sign(\lambda)\\ i\sign(\lambda)& 1\end{bmatrix}\right\}
d\lambda,\\
\K_2^{east}(\bt,\bs)&=
\frac{1}{4\pi}\sum_{n=-1}^1 \int_{-\infty}^\infty
\frac{e^{-4|\lambda|d}}{1-e^{-|\lambda|d}} e^{|\lambda|(x-x'-n\xi)}e^{i\lambda(y-y'-n\eta)}
\cdot\left\{\begin{bmatrix} n_1 & n_2 \\ n_2 & n_1+2i\sign(\lambda)n_2 \end{bmatrix}\right.\\
&\left.  - (|\lambda|n_1+i\lambda n_2)
\left(x-x'-n\xi-\frac{4-3e^{-|\lambda|d}}{1-e^{-|\lambda|d}}d\right)
\begin{bmatrix}1 & i\sign(\lambda)\\ i\sign(\lambda)& -1\end{bmatrix}\right\}
d\lambda.
\ea
\ee
Similarly, for singly periodic case,
\be
\ba
\K_1^{west}(\bt,\bs)&=
\sum_{m=-\infty}^{-2} \bD^{(S)}(\bt,\bs+\bl_{m0}), \\
&=\frac{1}{4\pi} \int_{-\infty}^\infty
\frac{e^{-2|\lambda|d}}{1-e^{-|\lambda|d}} e^{-|\lambda|(x-x')}e^{i\lambda(y-y')}
\cdot\left\{\begin{bmatrix} -n_1 & -n_2 \\ -n_2 & -n_1+2i\sign(\lambda)n_2 \end{bmatrix}\right.\\
&\left.  - (-|\lambda|n_1+i\lambda n_2)
\left(x-x'+\frac{2-e^{-|\lambda|d}}{1-e^{-|\lambda|d}}d\right)
\begin{bmatrix}-1 & i\sign(\lambda)\\ i\sign(\lambda)& 1\end{bmatrix}\right\}
d\lambda,\\
\K_1^{east}(\bt,\bs)
&=\sum_{m=2}^{\infty} \bD^{(S)}(\bt,\bs+\bl_{m0}), \\
&=\frac{1}{4\pi} \int_{-\infty}^\infty
\frac{e^{-2|\lambda|d}}{1-e^{-|\lambda|d}} e^{|\lambda|(x-x')}e^{i\lambda(y-y')}
\cdot\left\{\begin{bmatrix} n_1 & n_2 \\ n_2 & n_1+2i\sign(\lambda)n_2 \end{bmatrix}\right.\\
&\left.  - (|\lambda|n_1+i\lambda n_2)
\left(x-x'-\frac{2-e^{-|\lambda|d}}{1-e^{-|\lambda|d}}d\right)
\begin{bmatrix}1 & i\sign(\lambda)\\ i\sign(\lambda)& -1\end{bmatrix}\right\}
d\lambda.
\ea
\ee

\end{lemma}

\bibliographystyle{siam}
\bibliography{journalnames,fmm}

\begin{thebibliography}{10}

\bibitem{fmm2dlib}
{\sc T.~Askham, Z.~Gimbutas, L.~Greengard, L.~Lu, M.~O'Neil, M.~Rachh, and
  V.~Rokhlin}, {\em fmm2d software library}.
\newblock \url{https://github.com/flatironinstitute/fmm2d}, 2021.

\bibitem{barnett2010jcp}
{\sc A.~Barnett and L.~Greengard}, {\em A new integral representation for
  quasi-periodic fields and its application to two-dimensional band structure
  calculations}, Journal of Computational Physics, 229 (2010), pp.~6898--6914.

\bibitem{barnett2011bit}
\leavevmode\vrule height 2pt depth -1.6pt width 23pt, {\em A new integral
  representation for quasi-periodic scattering problems in two dimensions}, BIT
  Numerical mathematics, 51 (2011), pp.~67--90.

\bibitem{finufftlib}
{\sc A.~Barnett and J.~Magland}, {\em Non-uniform fast {F}ourier transform
  library of types $1$, $2$, $3$ in dimensions $1$, $2$, $3$}.
\newblock \url{https://github.com/ahbarnett/finufft}, 2018.

\bibitem{finufft}
{\sc A.~Barnett, J.~Magland, and L.~af~Klinteberg}, {\em A parallel non-uniform
  fast {F}ourier transform library based on an ``exponential of semicircle"
  kernel}, SIAM J. Sci. Comput., 41 (2019), pp.~C479--C504.

\bibitem{barnett2018cpam}
{\sc A.~H. Barnett, G.~R. Marple, S.~Veerapaneni, and L.~Zhao}, {\em A unified
  integral equation scheme for doubly periodic laplace and stokes boundary
  value problems in two dimensions}, Communications on Pure and Applied
  Mathematics, 71 (2018), pp.~2334--2380.

\bibitem{berman1994jmp}
{\sc C.~L. Berman and L.~Greengard}, {\em A renormalization method for the
  evaluation of lattice sums}, Journal of Mathematical Physics, 35 (1994),
  pp.~6036--6048.

\bibitem{bloch}
{\sc F.~Bloch}, {\em \"{U}ber die quantenmechanik der elektronen in
  kristallgittern}, Zeitschrift f\"{u}r Physik, 52 (1928), pp.~555--600.

\bibitem{ggq1}
{\sc J.~Bremer, Z.~Gimbutas, and V.~Rokhlin}, {\em A nonlinear optimization
  procedure for generalized {G}aussian quadratures}, SIAM J. Sci. Comput., 32
  (2010), pp.~1761--1788.

\bibitem{cheng1999jcp}
{\sc H.~Cheng, L.~Greengard, and V.~Rokhlin}, {\em A fast adaptive multipole
  algorithm in three dimensions}, J. Comput. Phys., 155 (1999), pp.~468--498.

\bibitem{cheng2006jcp}
{\sc H.~Cheng, J.~Huang, and T.~J. Leiterman}, {\em An adaptive fast solver for
  the modified {H}elmholtz equation in two dimensions}, J. Comput. Phys., 211
  (2006), pp.~616--637.

\bibitem{denlinger2017jmp}
{\sc R.~Denlinger, Z.~Gimbutas, L.~Greengard, and V.~Rokhlin}, {\em A fast
  summation method for oscillatory lattice sums}, Journal of Mathematical
  Physics, 58 (2017), p.~023511.

\bibitem{dienstfrey2001prsl}
{\sc A.~Dienstfrey, F.~Hang, and J.~Huang}, {\em Lattice sums and the
  two-dimensional, periodic green's function for the helmholtz equation},
  Proceedings of the Royal Society of London. Series A: Mathematical, Physical
  and Engineering Sciences, 457 (2001), pp.~67--85.

\bibitem{nufft2}
{\sc A.~Dutt and V.~Rokhlin}, {\em Fast {F}ourier transforms for nonequispaced
  data}, SIMA J. Sci. Comput., 14 (1993), pp.~1368--1393.

\bibitem{nufft3}
\leavevmode\vrule height 2pt depth -1.6pt width 23pt, {\em Fast {F}ourier
  transforms for nonequispaced data. {II}}, Appl. Comput. Harmon. Anal., 2
  (1995), pp.~85--100.

\bibitem{mckean}
{\sc H.~Dym and H.~P. McKean}, {\em Fourier Series and Integrals}, Academic
  Press, 1972.

\bibitem{enoch2001}
{\sc S.~Enoch, R.~McPhedran, N.~Nicorovici, L.~Botten, and J.~Nixon}, {\em Sums
  of spherical waves for lattices, layers, and lines}, Journal of Mathematical
  Physics, 42 (2001), pp.~5859--5870.

\bibitem{ewald}
{\sc P.~Ewald}, {\em Die berechnung optischer und elektrostatischer
  gitterpotentiale}, Annalen der Physik, 64 (1921), pp.~253--287.

\bibitem{fornberg1998practical}
{\sc B.~Fornberg}, {\em A practical guide to pseudospectral methods}, vol.~1,
  Cambridge university press, 1998.

\bibitem{gan2016sisc}
{\sc Z.~Gan, S.~Jiang, E.~Luijten, and Z.~Xu}, {\em A hybrid method for systems
  of closely spaced dielectric spheres and ions}, SIAM J. Sci. Comput., 38
  (2016), pp.~B375--B395.

\bibitem{fmps2013jcp}
{\sc Z.~Gimbutas and L.~Greengard}, {\em Fast multi-particle scattering: a
  hybrid solver for the {M}axwell equations in microstructured materials}, J.
  Comput. Phys., 232 (2013), pp.~22--32.

\bibitem{nufft6}
{\sc L.~Greengard and J.~Lee}, {\em Accelerating the nonuniform fast {F}ourier
  transform}, SIAM Rev., 46 (2004), pp.~443--454.

\bibitem{greengard1987jcp}
{\sc L.~Greengard and V.~Rokhlin}, {\em A fast algorithm for particle
  simulations}, J. Comput. Phys., 73 (1987), pp.~325--348.

\bibitem{greengard1997actanum}
{\sc L.~Greengard and V.~Rokhlin}, {\em A new version of the fast multipole
  method for the {L}aplace equation in three dimensions}, Acta. Numer., 6
  (1997), pp.~229--270.

\bibitem{hesthaven2007spectral}
{\sc J.~S. Hesthaven, S.~Gottlieb, and D.~Gottlieb}, {\em Spectral methods for
  time-dependent problems}, vol.~21, Cambridge University Press, 2007.

\bibitem{hrycak1998}
{\sc T.~Hrycak and V.~Rokhlin}, {\em An improved fast multipole algorithm for
  potential fields}, SIAM J. Sci. Statist. Comput., 19 (1998), pp.~1804--1826.

\bibitem{huang1999jmp}
{\sc J.~Huang}, {\em Integral representations of harmonic lattice sums},
  Journal of Mathematical Physics, 40 (1999), pp.~5240--5246.

\bibitem{jones}
{\sc D.~S. Jones}, {\em Generalised functions}, McGraw-Hill, New York, 1966.

\bibitem{nufft7}
{\sc J.~Lee and L.~Greengard}, {\em The type 3 nonuniform {FFT} and its
  applications}, J. Comput. Phys., 206 (2005), pp.~1--5.

\bibitem{linton2010}
{\sc C.~M. Linton}, {\em Lattice sums for the helmholtz equation}, SIAM Review,
  52 (2010), pp.~630--674.

\bibitem{liu2016jcp}
{\sc Y.~Liu and A.~H. Barnett}, {\em Efficient numerical solution of acoustic
  scattering from doubly-periodic arrays of axisymmetric objects}, Journal of
  Computational Physics, 324 (2016), pp.~226--245.

\bibitem{ggq2}
{\sc J.~Ma, V.~Rokhlin, and S.~Wandzura}, {\em Generalized {G}aussian
  quadrature rules for systems of arbitrary functions}, SIAM J. Numer. Anal.,
  33 (1996), pp.~971--996.

\bibitem{biros2015cicp}
{\sc D.~Malhotra and G.~Biros}, {\em {PVFMM}: a parallel kernel independent
  {FMM} for particle and volume potentials}, Commun. Comput. Phys., 18 (2015),
  pp.~808--830.

\bibitem{mcphedran2000jmp}
{\sc R.~McPhedran, N.~Nicorovici, L.~Botten, and K.~Grubits}, {\em Lattice sums
  for gratings and arrays}, Journal of Mathematical Physics, 41 (2000),
  pp.~7808--7816.

\bibitem{mikhlin}
{\sc S.~G. Mikhlin and S.~Prossdorf}, {\em Singular integral operators},
  Springer--Verlag, Berlin, 1986.

\bibitem{morse}
{\sc P.~Mores and H.~Feshbach}, {\em Methods of theoretical physics},
  McGraw-Hill, New York, 1953.

\bibitem{moroz2006}
{\sc A.~Moroz}, {\em Quasi-periodic {G}reen’s functions of the {H}elmholtz
  and {L}aplace equations}, J. Phys. A: Math. Gen., 36 (2006), p.~11247.

\bibitem{nisthandbook}
{\sc F.~W.~J. Olver, D.~W. Lozier, R.~F. Boisvert, and C.~W. Clark}, eds., {\em
  NIST Handbook of Mathematical Functions}, Cambridge University Press, May
  2010.

\bibitem{otani2008}
{\sc Y.~Otani and N.~Nishimura}, {\em A periodic {FMM} for {M}axwell’s
  equations in {3D} and its applications to problems related to photonic
  crystals}, Journal of Computational Physics, 227 (2008), pp.~4630--4652.

\bibitem{rayleigh}
{\sc L.~Rayleigh}, {\em On the influence of obstacles arranged in rectangular
  order upon the properties of a medium}, Philosophical Magazine, 34 (1892),
  pp.~481--502.

\bibitem{stakgold}
{\sc I.~Stakgold}, {\em Boundary value problems of mathematical physics},
  Macmillan, 1968.

\bibitem{trefethen2008gauss}
{\sc L.~N. Trefethen}, {\em Is gauss quadrature better than clenshaw--curtis?},
  SIAM review, 50 (2008), pp.~67--87.

\bibitem{wang2012convergence}
{\sc H.~Wang and S.~Xiang}, {\em On the convergence rates of legendre
  approximation}, Mathematics of Computation, 81 (2012), pp.~861--877.

\bibitem{wang2021}
{\sc J.~Wang, E.~Nazockdast, and A.~Barnett}, {\em An integral equation method
  for the simulation of doubly-periodic suspensions of rigid bodies in a
  shearing viscous flow}, Journal of Computational Physics, 424 (2021),
  p.~109809.

\bibitem{yan2018}
{\sc W.~Yan and M.~Shelley}, {\em Flexibly imposing periodicity in kernel
  independent {FMM}: A multipole-to-local operator approach}, Journal of
  Computational Physics, 335 (2018), pp.~214--232.

\bibitem{ggq3}
{\sc N.~Yarvin and V.~Rokhlin}, {\em Generalized {G}aussian quadratures and
  singular value decompositions of integral operators}, SIAM J. Sci. Comput.,
  20 (1998), pp.~699--718.

\bibitem{ying2004jcp}
{\sc L.~Ying, G.~Biros, and D.~Zorin}, {\em A kernel-independent adaptive fast
  multipole algorithm in two and three dimensions}, J. Comput. Phys., 196
  (2004), pp.~591--626.

\end{thebibliography}

\end{document}